\documentclass[a4paper,12pt,reqno]{amsart}
\usepackage{amssymb}
\usepackage{amsmath}
\usepackage{enumitem}
\usepackage{mathrsfs}
\usepackage[all]{xy}
\usepackage{mathtools}
\usepackage{hyperref}
\usepackage{url}

\usepackage[OT2,T1]{fontenc}
\DeclareSymbolFont{cyrletters}{OT2}{wncyr}{m}{n}
\usepackage[british]{babel}

\usepackage[margin=1.3in]{geometry}
\numberwithin{equation}{section} \numberwithin{figure}{section}

\DeclareMathOperator{\Pic}{Pic} 
\DeclareMathOperator{\Gal}{Gal} 
\DeclareMathOperator{\Aut}{Aut} 
\DeclareMathOperator{\Spec}{Spec}

\DeclareMathOperator{\characteristic}{char} \DeclareMathOperator{\Val}{Val}

\DeclareMathOperator{\Br}{Br} 

\DeclareMathOperator{\inv}{inv}

\DeclareMathOperator{\Bl}{Bl} \DeclareMathOperator{\Res}{R}

\DeclareMathSymbol{\Sha}{\mathalpha}{cyrletters}{"58}

\newcommand{\OO}{\mathcal{O}}

\newcommand{\PGL}{\textrm{PGL}}

\newcommand{\Dfive}{{\mathbf D}_5}
\newcommand{\Esix}{{\mathbf E}_6}

\newcommand\FF{\mathbb{F}}
\newcommand\PP{\mathbb{P}}
\newcommand\ZZ{\mathbb{Z}}
\newcommand\NN{\mathbb{N}}
\newcommand\QQ{\mathbb{Q}}

\newtheorem{lemma}{Lemma}

\newtheorem{theorem}[lemma]{Theorem}
\newtheorem{proposition}[lemma]{Proposition}
\newtheorem{corollary}[lemma]{Corollary}

\theoremstyle{definition}

\newtheorem{definition}[lemma]{Definition}
\newtheorem{remark}[lemma]{Remark}

\numberwithin{lemma}{section}

\begin{document}

\title[The Hasse principle for lines]
{The Hasse principle for lines on \\ del Pezzo surfaces}

\author{\sc J\"{o}rg Jahnel}
\address{J\"{o}rg Jahnel \\
Department Mathematik \\
Walter-Flex-Strasse 3\\
Universit\"{a}t Siegen \\
D-57072 \\
Siegen \\
Germany.}
\email{jahnel@mathematik.uni-siegen.de}
\urladdr{http://www.uni-math.gwdg.de/jahnel}

\author{\sc Daniel Loughran}
\address{Daniel Loughran \\
Leibniz Universit\"{a}t Hannover,
Institut f\"{u}r Algebra, Zahlentheorie
    und Diskrete Mathematik\\
Welfengarten 1\\
30167 Hannover\\
Germany.}
\email{loughran@math.uni-hannover.de}
\urladdr{http://www.iazd.uni-hannover.de/$\sim$loughran/}




\begin{abstract}
	In this paper, we consider the following problem: Does there 
	exist a cubic surface over $\QQ$ which contains no line over $\QQ$,
	yet contains a line over every completion of $\QQ$? This question may be
	interpreted as asking whether the Hilbert scheme of lines on a cubic surface
	can fail the Hasse principle. We also consider analogous problems, 
	over arbitrary number fields, for other del Pezzo surfaces and complete intersections of two quadrics.
\end{abstract}

\maketitle

\thispagestyle{empty}

\tableofcontents

\section{Introduction} \label{sec:intro}

One says that a class of varieties over a number field $k$ satisfies the 
Hasse principle if for each variety in the class, the existence
of a rational point over every completion of $k$ implies the existence
of a rational point over $k$. This principle takes its name from the 
classical Hasse-Minkowski theorem, which states that the Hasse principle holds 
for the class of quadric hypersurfaces. Already for cubic curves and cubic surfaces
however, the Hasse principle can fail. There has been 
much work on constructing and controlling such failures, and moreover there are many positive
results proving that the Hasse principle holds for special classes of varieties.
For an overview of various results and methods, see \cite{Sko01}.



Recently, there has been an interest in Hasse-type principles
for other existence problems in arithmetic geometry (e.g.~isogenies 
of elliptic curves \cite{Sut12}, \cite{BC14} and Apollonian gaskets \cite{BK14}).
In this paper we study Hasse-type principles
for linear subspaces on varieties. Namely, given a class of varieties
embedded into a fixed projective space $\PP^n$ over~$k$ and some $r \in \NN$, whether
the existence of a linear subspace of dimension~$r$ over every completion
of $k$ guarantees the existence of a linear subspace
of dimension~$r$ over $k$. Note that this is in fact
a special case of the usual Hasse principle. Indeed,
there is a Hilbert scheme parametrising the linear subspaces of fixed dimension,
and we are asking whether these schemes satisfy the Hasse principle.

For complete intersections of hypersurfaces of odd degree in $\PP^n$ over $k$, 
Birch \cite{Bir57} used the circle method to show that there always exists a linear subspace
over $k$ of given dimension $r$,
whenever $n$ is sufficiently large in terms of $r$ and the degrees of the
equations
(see \cite{Die10} and \cite{Bra14} for recent improvements). There does not appear however
to have been any work into such problems for varieties of small dimension.
For example, as already noted, the Hasse principle can fail
for cubic surfaces, but \emph{can the Hasse principle
for lines fail for cubic surfaces}?

Any smooth cubic surface over an algebraically closed field contains $27$
lines. Moreover the Hilbert scheme of lines is reduced, hence it is a $0$-dimensional
scheme of degree $27$. It is well-known that $0$-dimensional schemes
can fail the Hasse principle; counter-examples
occur  as the zero sets of monic integral polynomials with a root modulo
every integer, but no integer root (such polynomials automatically have a real root;
see Lemma~\ref{lem:Chebotarev}).
An explicit example here being
$$(x^2 - 2)(x^2 - 17)(x^2 - 34)=0.$$
See \cite{BB96}, \cite{Son08} and \cite{Son09} for other examples.
Our first result states that such schemes can indeed occur
as the Hilbert scheme of lines on a cubic surface.

\begin{theorem} \label{thm:cubic}
	Let $k$ be a number field. Then there exist smooth cubic surfaces
	over $k$ which fail the Hasse principle for lines.
\end{theorem}

Our proof is constructive. For example, an 
explicit counter-example $S \subset \PP^3$ over $\QQ$ is given by
\begin{align} \label{eqn:cubic}
\begin{split}
 &{- 5x^2w} - 5xy^2 - 2xyw + 5xz^2 - 9xzw - 5xw^2 + 9y^3 - 11y^2z + 29y^2w\\
 & + 43yz^2 - 52yzw - 4yw^2 - 13z^3 + 14z^2w - 96zw^2 + 45w^3=0.
\end{split}
\end{align}

The reader may easily verify this assertion herself with some help from a computer.
There is a univariate polynomial $f \in \QQ[t]$ of degree $27$, the roots of which 
rationally parametrise the set of the $27$ lines on the surface over $\bar \QQ$ (namely
$f(t)=0$ is isomorphic to the Hilbert scheme of lines on $S$). 
The polynomial $f$ may be determined explicitly, using Gr\"obner bases.
It turns out that $f$ factorises into $4$ irreducible polynomials
$f = f_1f_2f_3f_4$ of degrees $2, 5, 10$, and $10$, respectively. 
In particular, there is no $\QQ$-rational line on $S$.

The fields defined by $f_1$ and $f_2$ are
isomorphic to $\QQ(\sqrt{5})$ and $\QQ(\sqrt[5]{101})$, respectively. Hence $f_1$ 
has a root over $\QQ_p$ if $p \equiv 1,4 \bmod 5$ or $p=\infty$, and $f_2$ has a root over 
$\QQ_p$ if $p \equiv 2,3,4 \bmod 5$ or $p=5$.
Thus $S$ contains a line over every completion of $\QQ$, i.e.~it is a counter-example
to the Hasse principle for lines. We will explain in 
the main body of this article how we found this surface, and give
a general method for constructing similar counter-examples.
To the authors' knowledge, this is the first known counter-example
to the Hasse principle for the Hilbert scheme of linear subspaces of positive
dimension on a hypersurface.

Note that a cubic surface contains a line over $k$ if and only if it 
contains a conic over $k$. We therefore obtain the following amusing corollary to Theorem~\ref{thm:cubic}.
\begin{corollary}
	Let $k$ be a number field. Then there exist smooth cubic surfaces
	over $k$ which fail the Hasse principle for conics.
\end{corollary}
Again to the authors' knowledge, this is the first result
of its kind in the literature. We are also able to prove results
about the distribution of counter-examples.
To state this, note that the coefficients of a cubic surface
determine a point in $\PP^{19}$. We therefore define
the height of a cubic surface over $k$ to be the usual naive height
of the corresponding point in $\PP^{19}$ (see e.g.~\cite[\S 2]{Ser97a}).

\begin{theorem} \label{thm:thin_cubic}
	Let $k$ be a number field. 
	\begin{enumerate}
		\item[(a)] The collection of smooth cubic surfaces
		over $k$ which fail the Hasse principle for lines is Zariski dense
		in $\PP^{19}$.
		\item[(b)] 	Nevertheless, when ordered by the height of their coefficients,
		$100 \%$ of smooth cubic surfaces satisfy the Hasse principle for lines.
	\end{enumerate}
\end{theorem}
We prove Zariski density by showing that the set of the reductions of the counter-examples
modulo suitable prime ideals of $k$ is rather large.
The second part of this theorem follows from the fact
that $100 \%$ of cubic surfaces have an irreducible Hilbert scheme
of lines, together with the fact that irreducible finite \'{e}tale
schemes satisfy the Hasse principle (see Lemma \ref{lem:HP}).

Cubic surfaces belong to a natural class of surfaces, namely they are
the del Pezzo surfaces of degree $3$.
Our next result solves the analogous problem for del Pezzo surfaces in general,
where one has a natural notion of a line (see \S \ref{sec:lines}).

\begin{theorem} \label{thm:main}
	Let $k$ be a number field and let $1 \leq d \leq 9$.
	Then the class of del Pezzo surfaces of degree $d$ fails
	the Hasse principle for lines if and only if
	$$d = 8,5,3,2 \text{ or } 1.$$
\end{theorem}

Our results are actually stronger than stated here.
Namely, we classify all possible cases where counter-examples can occur
in terms of Galois actions on lines, and show that, at least when
$d \geq 3$, counter-examples of these types actually exist.
We also prove an analogue of Theorem \ref{thm:thin_cubic}, by showing
that the collection of counter-examples is a Zariski dense, yet thin, subset of a suitable
Hilbert scheme (see Theorem \ref{thm:distribution}). This is one of the 
main results in this paper, and it says that whilst there are many counter-examples,
they are still in some sense quite rare. 
See \S \ref{sec:main} and \S \ref{sec:thin} for our complete results on del Pezzo surfaces.

We now explain the ideas behind the proof of Theorem \ref{thm:main}. The cases
of degrees $9$ and $8$ are handled using some classical geometric methods.
For the other cases, we use the Chebotarev density theorem to 
give an explicit criterion on the Galois group of the 
splitting field of the surface, in terms of its action
on the graph of lines of the surface, for the surface to be a counter-example
to the Hasse principle for lines. This reduces
the problem to a purely group theoretic one about classifying
certain group actions on certain graphs, which we perform
using a computer.
In the case of degree $7,6$ or $4$, one finds that this criterion
is never satisfied, hence the Hasse principle for lines always holds.

In degrees $5,3,2$ and $1$, we find explicit groups
which satisfy our criterion. A priori however, there is no guarantee 
that there exist del Pezzo surfaces with these Galois actions over $k$
(this is the ``inverse Galois problem'' for del Pezzo surfaces).
In degrees $5$ and $3$, there are very few cases,
and moreover the surfaces which arise this way are rational.
Thus we construct surfaces with these Galois actions
by performing blow-ups and blow-downs of configurations
of closed points in $\PP^2$. These geometrical constructions yield surfaces
with a line over all but finitely many completions of $k$, yet no line over~$k$.
To obtain counter-examples to the Hasse principle for lines, we need quite strict control
over the decomposition groups of the splitting fields
of the blown-up points. When the Galois group is solvable, this is achieved using results of Sonn \cite{Son08}.
The only non-solvable group which occurs is $S_5$, and here we use
work of Kedlaya on the construction of quintic fields with square-free discriminants
\cite{Ked12}. For del Pezzo surfaces of degrees $2$ and $1$ many more cases
arise. We therefore illustrate the range of behaviour which occurs
by constructing both rational and minimal counter-examples (i.e.~surfaces
for which no Galois orbit of lines can be contracted). These latter
surfaces arise as conic bundles.


Note that Theorem \ref{thm:main} implies that del Pezzo surfaces
of degree $4$ satisfy the Hasse principle for lines, something which
is by no means obvious. The authors were able to find a non-computer-assisted proof of this
result, and moreover generalise it to intersections of two quadrics of even dimension.

\begin{theorem} \label{thm:two_quadrics}
	Let $n \geq 0$ and let $k$ be a number field.
	Let $X$ be a smooth $2n$-dimensional complete intersection of two quadrics over $k$.
	If $X$ contains no $n$-plane over $k$, then it contains
	no $n$-plane over infinitely many completions of~$k$.
	In particular, $X$ satisfies the Hasse principle for $n$-planes.
\end{theorem}
The planes considered in Theorem \ref{thm:two_quadrics} are in some
respects the most interesting, since they are the linear subspaces
of maximal dimension which can occur. There are various
results already in the literature on the existence of linear subspaces on intersections
of quadrics. For example, an application of the version of the circle method used in \cite{Bra14}
would prove the Hasse principle for $r$-planes when $r$ has order of magnitude $O(\sqrt{n})$ and $k=\QQ$.
Much more is known when~$k$ is totally imaginary; namely, as explained
in \cite[(2.11)]{DW03}, one may combine \cite[Cor.~10.4]{CTSSD87}, \cite[Cor.~2.4]{Lee84} and 
\cite[Thm.~1]{Mar97} to obtain
the existence of an $r$-plane here for $r$ of the size $2n/3 + O(1)$.
In particular, Theorem~\ref{thm:two_quadrics} is out of reach
from what these methods are currently able to achieve.
It is proved by showing that the Hilbert scheme of
$n$-planes on $X$ is a torsor for a group scheme whose Tate-Shafarevich
group is trivial. 

The analogue of Theorem \ref{thm:two_quadrics} fails for
smooth $(2n+1)$-dimensional complete intersections of two quadrics over $k$
(the linear subspaces of maximal dimension being $n$-planes). Here the Hilbert
scheme of $n$-planes is a torsor for  the intermediate Jacobian of $X$.
For $n=0$, we obtain a curve of genus $1$, and it is well-known that these can fail
the Hasse principle (see \cite[p.~164]{Sko01} for an explicit example over $\QQ$).
In private communication, Damiano Testa has informed the authors 
that he is able to construct counter-examples of dimension $3$ over $\QQ$.

\subsection*{Outline of the paper} In Section \ref{sec:0-dim} we gather results
on the Hasse principle for finite \'{e}tale schemes and the structure of the Hilbert
scheme of lines on del Pezzo surfaces. Theorem \ref{thm:main} is then
proved in Section \ref{sec:main}. Theorem \ref{thm:thin_cubic}, together
with its generalisation Theorem \ref{thm:distribution}, is proved in Section \ref{sec:thin}.
Section \ref{sec:conceptual} is dedicated to the proof of Theorem \ref{thm:two_quadrics}.

\subsection*{Acknowledgements}
The authors have benefited from conversations with
Julia Brandes, Jean-Louis Colliot-Th\'{e}l\`{e}ne, Rainer Dietmann, 
Jordan Ellenberg, Andreas-Stephan Elsenhans, Christopher Frei, David Speyer, Anthony V\'{a}rilly-Alvarado
and Trevor Wooley. All computations are with \texttt{Magma} \cite{BCP97}.

\subsection{Notation} \label{sec:notation}
For a scheme $X$ of finite type over a field $k$, we denote by $\pi_0(X)$ the scheme of
connected components of $X$ (see \cite[Def.~10.2.19]{Liu02}). For a field extension
$k \subset K$, we denote by $X_K = X \times_k K$ the base change of $X$ to $K$.

For a number field $k$, we denote its ring of integers by $\OO_k$.
For a place $v$ of $k$, we shall denote by $k_v$
the completion of $k$ with respect to $v$. For a non-zero prime ideal $\mathfrak{p}$
of $k$, we denote by $\FF_\mathfrak{p}$ the corresponding residue field
and by $\OO_{k,\mathfrak{p}}$ the \emph{localisation} of $\OO_k$ at $\mathfrak{p}$.

Bold math letters (e.g.~$\mathbf{A}_n$) are used for root
systems, while for groups we use non-bold math letters
(e.g. $A_n$ is the alternating group on $n$ letters).

\section{\texorpdfstring{$0$}{0}-dimensional schemes and del Pezzo surfaces} \label{sec:0-dim}
We begin with some results on
the arithmetic of finite \'{e}tale schemes. 
Our interest in these
stems from the fact that they occur as the Hilbert schemes of lines
on del Pezzo surfaces. We finish this section by gathering some results 
on the geometry of del Pezzo surfaces, together with results on 
their Hilbert schemes and graphs of lines.

\subsection{Finite \'{e}tale schemes} 
Let $X$ be a finite \'{e}tale scheme over a perfect field~$k$
(i.e.~$X$ is a finite reduced $0$-dimensional scheme over $k$). 
We define the degree of $X$ to be 
$$\deg X = \sum_{x \in X} [k(x):k],$$
where the sum is over all points $x$ of $X$,
and $k(x)$ denotes the residue field of $x$.
We shall say that such a scheme $X$ is \emph{split} if 
$$\#X(k) = \deg X.$$
The smallest field $K$ over which $X$ becomes split is called 
the splitting field of~$X$. This is a finite Galois extension of $k$.

\subsubsection{A dictionary}

There is a useful well-known dictionary for such schemes.
Namely, let $K/k$ be a finite Galois extension with Galois group $\Gamma$
and let $X$ be a finite \'{e}tale scheme over $k$ with splitting field $K$. 
Firstly, to $X$ one may associate the $\Gamma$-set (i.e.~a set
with a left action of $\Gamma$) given by $X(K)$ equipped with its natural $\Gamma$-action.
Secondly, we may also associate to $X$ the finite \'{e}tale $k$-algebra
$$ \prod_{x \in X} k(x).$$
These constructions give rise to bijections between the following sets:
\begin{enumerate}
	\item The set of $k$-isomorphism classes of finite \'{e}tale schemes 
	over $k$ with splitting field $K$.
	\item The set of $k$-isomorphism classes of finite \'{e}tale $k$-algebras with splitting field $K$.
	\item The set of $\Gamma$-isomorphism classes of finite $\Gamma$-sets with a faithful 
	action of~$\Gamma$.
\end{enumerate}
See \cite[Thm.~V.6.4]{Bou81} and \cite[Prop.~V.10.12]{Bou81}.
We shall frequently use this dictionary to pass between 
these different objects in this paper. One trivial, though key, property
is that $X$ admits a $k$-point if and only if $\Gamma$ acts with a fixed point
on the associated $\Gamma$-set $X(K)$.

\subsubsection{The Hasse principle} \label{sec:HP}

We now gather some mostly well-known results on the Hasse principle for finite \'{e}tale
schemes, via the use of our dictionary. Throughout the remainder of this section $k$
is a number field.  We begin with a criterion for 
local solubility at a given place.

\begin{lemma}\label{lem:decomposition}
	Let $X$ be a finite \'{e}tale scheme over $k$. Let $K/k$
	denote the splitting field, with Galois group $\Gamma$, and let 
	$v$ be a place of $k$. Then $X$ admits a $k_v$-point if and only for
	any place $w$ of $K$ above $v$, the decomposition group of $w$ acts with a
	fixed point on the associated $\Gamma$-set $X(K)$.
\end{lemma}
\begin{proof}
	The splitting field of $X_{k_v}$ is $K_w$,
	and the Galois group $\Gal(K_w/k_v)$ may be identified
	with the decomposition group of $w$. As
	$X$ admits a $k_v$-point if and only if this
	group acts with a fixed point, the result is proved.
\end{proof}

Using this, we give a criterion for solubility at almost all places.

\begin{lemma}\label{lem:Chebotarev}
	Let $X$ be a finite \'{e}tale scheme over $k$. Let $K/k$
	denote the splitting field, with Galois group $\Gamma$. Then $X$ is locally soluble
	at all but finitely many places of $k$, but not soluble over $k$,
	if and only if on the associated
	$\Gamma$-set $X(K)$, each conjugacy class of $\Gamma$ acts with a fixed point,
	but the group $\Gamma$ acts without a fixed point.
	
	In which case, for each place $v$ of $k$ which is either archimedean
    or unramified in~$K$, the scheme $X$ admits a $k_v$-point.
\end{lemma}
\begin{proof}
	The Chebotarev density theorem \cite[Thm.~13.4]{Neu99}
	implies that every cyclic subgroup of $\Gamma$
	occurs as a decomposition group for infinitely many
	places of $k$. Hence if there is an element of $\Gamma$
	which acts without a fixed point, then by Lemma \ref{lem:decomposition} there are infinitely many
	places $v$ of $k$ for which $X$ has no $k_v$-point.
	If however every conjugacy class of $\Gamma$ acts with a fixed point,
	then an application of Lemma \ref{lem:decomposition} shows that 
	$X$ admits a $k_v$-point for every place $v$ of $k$
	which is either archimedean or unramified in $K/k$. This completes the proof.
\end{proof}
We now present some simple cases in which the Hasse principle holds.

\begin{lemma} \label{lem:HP}
	Let $X$ be a finite \'{e}tale scheme over $k$ and let $\Gamma$ denote
	the Galois group of the splitting field of $X$. Assume that $X$
	is irreducible or that $\Gamma$ is cyclic. If $X(k) = \emptyset$, then
	$X(k_v) = \emptyset$ for infinitely many places $v$ of $k$.
	In particular $X$ satisfies the Hasse principle.
\end{lemma}
\begin{proof}
	When $X$ is irreducible this is well-known, and follows,
	for example, from Lemma \ref{lem:Chebotarev} and a theorem
	of Jordan \cite[Thm.~4]{Ser03}. When $\Gamma$ is cyclic,
	the result is an easy application of Lemma \ref{lem:Chebotarev}
	(if a generator of $\Gamma$ acts with a fixed point, then so does $\Gamma$).
\end{proof}

\begin{remark}
	For non-cyclic groups however, counter-examples to the Hasse principle always occur.
	Namely, let $k$ be a number field and let $\Gamma$ be a non-cyclic group which is realisable
	as a Galois group over $k$. Then it was shown in \cite{Son08}
	and \cite{Son09} that there exists some finite \'{e}tale scheme over $k$ whose splitting
	field has Galois group $\Gamma$ and which is a counter-example
	to the Hasse principle.
	
	In the solvable case, these results will suffice for
	our purposes. In the non-solvable case however,
	Sonn's construction does not work for us since the schemes produced have very large degree 
	in general. We therefore take a different approach, using work of Kedlaya \cite{Ked12}.
\end{remark}

We next give a condition for the failure of the Hasse
principle in terms of the associated \'{e}tale algebra.

\begin{lemma} \label{lem:Hasse}
	Let $K_1,\ldots,K_r$ be number fields which
	contain, yet are strictly larger than,
	a fixed number field $k$. Let
	$$X = \Spec K_1 \sqcup \cdots \sqcup \Spec K_r,$$
	considered as a scheme over $k$.
	Then $X$ fails the Hasse principle if and only if 
	for each prime ideal $\mathfrak{p}$ of $k$, one of the fields $K_i$ 
	admits an unramified prime ideal above $\mathfrak{p}$
	of inertia degree $1$.
\end{lemma}
\begin{proof}
	By construction $X$	admits no $k$-point. The key relevant fact for the proof is that 
	if $F/E$ is an unramified extension
	of non-archimedean local fields of inertia degree $1$, then $F=E$
	(see \cite[\S III.5]{Ser80}). 
	This easily implies that $X$ is soluble at each non-archimedean
	place of $k$ if and only if for each prime ideal $\mathfrak{p}$ of $k$,
	one of the fields $K_i$ admits an unramified prime ideal above $\mathfrak{p}$
	of inertia degree $1$. It follows from Lemma \ref{lem:Chebotarev} however
	that $X$ is soluble at each non-archimedean
	place of $k$ if and only if it is soluble
	at \emph{all} places of $k$, and the result follows.
\end{proof}

\begin{remark}
	M.~Stoll \cite[Prop.~5.12]{Sto07} has shown that the finite abelian
	descent obstruction is the only one to the Hasse principle
	for finite \'{e}tale schemes. It is also quite simple to see that the Brauer-Manin
	obstruction is the only one to the Hasse principle in this case.
	Indeed, let $X$ be a finite \'{e}tale scheme over~$k$
	which fails the Hasse principle and let $(x_v) \in X(\mathbb{A}_k)$
	be an adelic point.	Choose an irreducible component $X_0$ of $X$ for which the
	set $$S= \{ v \in \Val(k) : x_v \in X_0(k_v) \},$$
	contains a non-complex place $v_0$. Note that
	the complement of $S$ in $\Val(k)$ is infinite, by Lemma \ref{lem:HP}.
	Write $X_0 = \Spec K$, so that $x_{v_0}$ corresponds to some unramified
	place $w_0 \mid v_0$ of $K$ of inertia degree $1$. Then by class field theory,
	there exists $b_0 \in \Br X_0=\Br K$ such that $\inv_{w_0} b_0 \neq 0$,
	but such that $\inv_w b_0 = 0$ for all other $w \in \Val(K)$ above
	the places of $S$. Hence the adelic point $(x_v)$ is not orthogonal
	to the Brauer group element
	$$b = (b_0,0) \quad \in \Br X_0 \oplus \Br (X \setminus X_0) = \Br X.$$
	Thus $X(\mathbb{A}_k)^{\Br} = \emptyset$, as required.
\end{remark}

\subsubsection{Existence of counter-examples} \label{sec:existence}
We finish by gathering some results on the existence of finite \'{e}tale
schemes which fail the Hasse principle.
The first result says that in the solvable
case, we may choose our field extensions such that the solubility conditions at
the ramified primes in Lemma \ref{lem:Chebotarev} are always satisfied.

\begin{lemma} \label{lem:solvable}
	Let $k$ be a number field and let $\Gamma$
	be a solvable group. Then there exists a Galois
	extension $K/k$ with Galois group $\Gamma$
	with the following property:
	
	Let $X$ be a finite \'{e}tale
	$k$-scheme with splitting field $K$, such that 
	each conjugacy class of $\Gamma$ acts with a fixed 
	point on the associated $\Gamma$-set $X(K)$, 
	but $\Gamma$ acts without a fixed point.
	Then $X$ is a counter-example to the Hasse principle.
\end{lemma}
\begin{proof}
	It is shown in the proof of \cite[Thm.~2]{Son08} 
	that there exists a finite Galois extension $K/k$ with
	Galois group $\Gamma$, all of whose decomposition groups are cyclic. 
	The result then follows from Lemma \ref{lem:decomposition}
	and Lemma \ref{lem:Chebotarev}.
\end{proof}

We also construct explicit counter-examples when $\Gamma = \ZZ/2\ZZ \times \ZZ/2\ZZ$.

\begin{lemma} \label{lem:biquadratic}
	Let $k$ be a number field. Then there exist $a,b \in k^*$
	such that the $k$-scheme
	$$(x^2 - a)(x^2 - b)(x^2 - ab) = 0,$$
	fails the Hasse principle.
\end{lemma}
\begin{proof}
	Choose rational primes $p$ and $q$ such that $p \equiv 1 \pmod 8,
	q \equiv 1 \pmod p$ and 
	$$\QQ(\sqrt{p}), \QQ(\sqrt{q}), \QQ(\sqrt{pq}) \not \subset k.$$
	This is possible since there are infinitely many such primes
	by Dirichlet's theorem, and since $k$ contains only finitely many subfields
	by Galois theory. A simple computation shows that the scheme
	$$(x^2 - p)(x^2 - q)(x^2 - pq) = 0,$$
	over $\QQ$ is a counter-example to the Hasse principle, and that it
	also remains a counter-example to the Hasse principle upon base-change to $k$.
\end{proof}

\subsection{Del Pezzo surfaces}
In this section we gather various facts we shall require about 
the geometry of del Pezzo surfaces. For background on del Pezzo
surfaces, we use \cite{Man86}, \cite[Ch.~8]{Dol12} and \cite[III.3]{Kol96}.
\subsubsection{Classification} \label{sec:classification}
\begin{definition}
	Let $k$ be a field. A del Pezzo surface over $k$ is a smooth projective surface $S$ 
	over $k$ with ample anticanonical divisor $(-K_S)$. The degree $d$ of $S$ is defined to be
	the self-intersection number $d=(-K_S)^2$.
\end{definition}

The degree $d$ of any del Pezzo surface $S$ satisfies $1 \leq d \leq 9$.
If $k$ is algebraically closed,
any del Pezzo surface is isomorphic to either $\PP^2$,
$\PP^1 \times \PP^1$, or the blow-up of $\PP^2$ in $9-d$ points
in general position $(d \leq 8$). Here by general position, we mean that 
no three lie on a line, no six
lie on a conic, and no eight lie on a cubic which is singular at one
of the points. In which case,
the exceptional curves on $S$ consist of the exceptional
curves of the  blow-ups of the points, together with the 
strict transforms of the following curves (see \cite[Thm.~26.2]{Man86}).
\begin{enumerate}
	\item Lines through two of the points.
	\item Conics through five of the points.
	\item Cubic curves through seven of the points, with a double point at exactly one of them.
	\item Quartic curves through eight of the points, with a double point at exactly three of them.
	\item Quintic curves through eight of the points, with a double point at exactly six of them.
	\item Sextic curves through eight of the points, with a triple point at exactly one of them
	and double points at the rest.
\end{enumerate}

For general $k$, one may construct
del Pezzo surfaces by blowing-up collections of \emph{closed} points
of $\PP^2$ in general position. Here we say
that a collection of closed points is in general position if the corresponding
points over $\bar k$ are in general position. Note that, in general, not every del
Pezzo surface arises in this way, as surfaces constructed in this manner will
always be non-minimal.

\subsubsection{The lines} \label{sec:lines}

Much of the arithmetic and geometry of del Pezzo surfaces is governed by the lines on 
the surface. We now let $k$ be a perfect field.

\begin{definition} \label{def:line}
	Let $S$ be a del Pezzo surface of degree $d$ over $k$ and let $L$ 
	be a curve on $S$ defined over $k$. Then we call $L$ a \emph{line} if:
	\begin{enumerate}
		\item $L\cdot(-K_S) = 3$, if $d=9$.
		\item $L\cdot(-K_S) = 2$, if $S_{\bar k} \cong \PP^1 \times \PP^1$.
		\item $L \cdot(-K_S) = 1$ and $p_a(L)=0$, otherwise.
	\end{enumerate}
\end{definition}

When $d\leq 7$ such curves are exactly the $(-1)$-curves on $S$. In the other cases
we have chosen a definition which reflects the geometry over the algebraic
closure. Indeed, a curve $L$ in $\PP^2$ with $L\cdot(-K_{\PP^2}) = 3$ is exactly a line
in the usual sense (since $-K_{\PP^2} = 3L$), and moreover a curve $L$ in $\PP^1 \times \PP^1$ with $L\cdot(-K_{\PP^1 \times \PP^1}) = 2$ 
is also exactly a line in the usual sense, once one embeds $\PP^1 \times \PP^1$
into $\PP^3$ as a quadric surface via the Segre embedding. Note that del
Pezzo surfaces of degree~$1$ contain curves $C$ such that $C \cdot (-K_S)=1$
but such that $C$ is not a line; any curve in the anticanonical linear
system is such a curve.

For a del Pezzo surface $S$, we shall denote by $\mathcal{L}(S)$ the Hilbert scheme of lines
on $S$. This is the scheme which parametrises the lines inside $S$,
whose existence is a special case of a general result of Grothendieck
\cite[Expos\'{e} 221, Thm.~3.1]{FGA}.

\begin{proposition} \label{prop:Hilbert}
	Let $S$ be a del Pezzo surface of degree $d$ over $k$.
	Then $\mathcal{L}(S)$ is a reduced projective scheme over $k$, which satisfies
	the following:
	\begin{enumerate}
		\item $\mathcal{L}(S)_{\bar k} \cong \PP^2$, if $d=9$. \label{item:P2}
		\item $\mathcal{L}(S)_{\bar k} = \PP^1 \sqcup \PP^1$,
			if $S_{\bar k} \cong \PP^1 \times \PP^1$. \label{item:P1_times_P1}
		\item $\mathcal{L}(S)$ is a finite \'{e}tale scheme of degree	\label{item:general}
			  \begin{table}[htp]
			  \centering
				\begin{tabular}{c|c|c|c|c|c|c|c|c}
					$d$&$8$&$7$&$6$&$5$&$4$&$3$&$2$&$1$ \\
					\hline
					$\deg \mathcal{L}(S)$&$1$&$3$&$6$&$10$&$16$&$27$&$56$&$240$ \\
				\end{tabular}
			  \label{tab:overview} ,
			  \end{table}\\
			otherwise.
	\end{enumerate}
\end{proposition}
\begin{proof}
	The fact that $\mathcal{L}(S)$ is reduced is well-known, see for instance
	\cite[Prop.~3.5]{Sch85}. 
	As the formation of the Hilbert scheme
	$\mathcal{L}(S)$ respects base-change, we may assume that $k$ is algebraically closed.
	In which case, \eqref{item:P2} follows from the classical duality of lines in the projective plane,
	and \eqref{item:P1_times_P1} follows from the fact that $\PP^1 \times \PP^1$ admits 
	exactly two distinct rulings by lines. For \eqref{item:general}, 
	the calculation of the degree of $\mathcal{L}(S)$ follows from the
	calculation of the number of lines on del Pezzo surfaces over algebraically
	closed fields \cite[Cor.~25.5.4]{Man86}.
\end{proof}
We shall say that a del Pezzo surface $S$ of degree $d \leq 7$ over $k$ is \emph{split} if
$\mathcal{L}(S)$ is split (as a finite \'{e}tale scheme).
There is a unique smallest 
Galois extension over which $S$ becomes split, called the splitting field
of $S$.

We shall say that a class of del Pezzo surfaces over a number field $k$
satisfies the \emph{Hasse principle for lines}, if for each surface $S$ in the class,
the Hilbert scheme $\mathcal{L}(S)$ satisfies the Hasse principle.

\subsubsection{Graphs of lines}
Let $d \leq 7$ and let $S$ be a del Pezzo surface of degree~$d$ over an algebraically
closed field $k$. Denote by $G(S)$ the \emph{graph of lines} 
of $S$ (see \cite[\S 26.9]{Man86}). This has one vertex for each line on $S$, and given two distinct lines
$L_1$ and $L_2$ on $S$, there is an edge of multiplicity $L_1 \cdot L_2$
between the corresponding vertices. Note that multiple edges occur
if and only if $d\leq 2$.

There exists a graph $G_d$ such that if $S$ is a del Pezzo surface
of degree $d$, then $G(S) \cong G_d$. This graph may be defined via
the root system $\mathbf{E}_{9-d}$ (see \cite[\S 8.2.2]{Dol12}). Namely, there is one vertex for each lattice point $\ell \in 
\mathbf{E}_{9-d}$ which satisfies $(\ell,\ell)=-1$, and given two distinct such points
$\ell_1$~and~$\ell_2$, there is an edge of multiplicity $(\ell_1,\ell_2)$
between the corresponding vertices.

We define the automorphism group $\Aut(G_d)$ of $G_d$
to be those automorphisms of the vertex set
of $G_d$ which preserve edge multiplicities between vertices
(in particular we are not interested in ``automorphisms''
of $G_d$ which fix every vertex, yet permute multiple edges
between adjacent vertices).
The automorphism group of $G_d$ is 
the Weyl group $W(\mathbf{E}_{9-d})$ \cite[Thm.~23.8]{Man86}.
We have the following explicit descriptions.
\begin{itemize}
	\item $G_7$: Path graph on $3$ vertices. $W(\mathbf{E}_{2}) \cong W(\mathbf{A}_2) \cong \ZZ/2\ZZ$.
	\item $G_6$: Cycle graph on $6$ vertices. $W(\mathbf{E}_{3}) \cong W(\mathbf{A}_{1} \times \mathbf{A}_{2}) \cong D_6$.
	\item $G_5$: The Petersen graph. $W(\mathbf{E}_{4}) \cong W(\mathbf{A}_{4}) \cong S_5$.
	\item $G_4$: The Clebsch graph. $W(\mathbf{E}_{5}) \cong W(\mathbf{D}_{5}) \cong (\ZZ/2\ZZ)^4 \rtimes S_5$.
	\item $G_3$: The (dual of the) Schl\"{a}fli graph. $W(\mathbf{E}_{6})$.
\end{itemize}

Let now $S$ be a del Pezzo surface of degree $d$ over a perfect field $k$
with splitting field $K$ and $\Gamma = \Gal(K/k)$. Here we define $G(S)=G(S_{\bar k})$.
The action of $\Gamma$ on the lines of $S_{\bar k}$
induces an action of $\Gamma$ on $G(S)$. The vertices therefore form
a $\Gamma$-set, whose associated finite \'{e}tale
scheme is exactly the Hilbert scheme $\mathcal{L}(S)$ of lines of $S$. 
Choosing a graph isomorphism $G(S) \cong G_d$, we obtain an action
of $\Gamma$ on $G_d$, which identifies $\Gamma$ with a subgroup of 
$\Aut(G_d)=W(\mathbf{E}_{9-d})$. Different choices of the graph isomorphism give rise to conjugate
subgroups, so we often view this as an action of $\Gamma$ on $G_d$, well-defined
up to conjugacy. If $O_1,\ldots,O_r$ denote the orbits of the action of $\Gamma$
on the vertices of $G(S)$, we call the unordered list $[\# O_1,\ldots,\# O_r]$
the \emph{orbit type} of $S$.

\section{The Hasse principle for lines on del Pezzo surfaces} \label{sec:main}
The aim of this section is to prove Theorem \ref{thm:main}.

\subsection{Proof strategy}
The cases $d=8$ and $d=9$ shall be handled
separately, using basic geometric and arithmetic facts, together with 
Proposition \ref{prop:Hilbert}. 
For $d\leq 7$, the following proposition is the key of our strategy.

\begin{proposition} \label{prop:criterion}
	Let $S$ be a del Pezzo surface of degree $d \leq 7$ over a number field~$k$.
	Let $K/k$ be the splitting field of $S$ with Galois group $\Gamma$.
	Then $S$ contains a line over almost all completions of $k$, but no line over $k$, if and only 
	if each conjugacy class of $\Gamma$ acts with a fixed point on the graph $G_d$,
	but $\Gamma$ acts without a fixed point on $G_d$.
	
	In which case, for every place $v$ of $k$ which is either archimedean or
	unramified in~$K$, the surface $S$ admits a line over $k_v$.
\end{proposition}
\begin{proof}
	This follows immediately from Lemma \ref{lem:Chebotarev}.
\end{proof}

Using the criterion given by Proposition \ref{prop:criterion}, the first step is purely
computational. Namely for each $d\leq 7$, we enumerate all conjugacy classes
of subgroups of $\Aut(G_d)= W(\mathbf{E}_{9-d})$. We store those subgroups $\Gamma$ which
have no fixed point when acting on $G_d$, yet are such that each
conjugacy class of $\Gamma$ admits a fixed point. This
gives a list for the Galois groups of the splitting
fields of del Pezzo surfaces which could fail the Hasse principle
for lines, together with the associated Galois actions on the lines of the surface.
Note that by Lemma \ref{lem:HP}, such subgroups must be non-transitive and non-cyclic.
This classification can be done easily by hand for $d=7$ and $6$,
but such computations become increasingly more difficult,
so for $d\leq 5$ we use a computer.

Fix now some $d\leq 7$ and suppose that we have found such a
subgroup $\Gamma$ acting on $G_d$.
For each field extension $K/k$ with Galois group $\Gamma$,
we may associate via our dictionary
a finite \'{e}tale $k$-scheme $X$.
Note that by construction, this scheme will
be soluble at all those places of $k$ which are either archimedean or unramified in $K$.
The first step is to find such a field $K$ for which the associated scheme $X$
fails the Hasse principle, by taking care of the 
conditions at the ramified places. 
We do this using the results obtained 
in \S \ref{sec:HP} and \S \ref{sec:existence}.
In the most difficult case of $S_5$,
we also require work of Kedlaya \cite{Ked12}.

The next step is the ``inverse Galois problem'' for del Pezzo surfaces,
namely to find a del Pezzo surface $S$
whose Hilbert scheme of lines is isomorphic to the scheme $X$ found above.
We are fortunate since when $d \geq 3$, the cases which arise
correspond to rational surfaces. We therefore
construct them by blow-ups of $\PP^2$ followed by certain 
blow-downs. Only when $d=2$ or $1$ do non-rational examples occur,
and here we construct counter-examples using conic bundles.

\subsubsection{Implementing the strategy}
In the interests of clarity, we now explain how to
construct the graphs $G_d$, together with the associated 
action of $W(\mathbf{E}_{9-d})$, and how one runs the computational part of the proof strategy
in \texttt{Magma}. We do this now in the case of 
the graph $G_2$, explaining briefly at the end what changes need to be
made to construct the other graphs. The following commands
\begin{verbatim}
	R_E7 := RootDatum("E7");
	Cox_E7 := CoxeterGroup(R_E7);
	W_E7 := StandardActionGroup(Cox_E7);
\end{verbatim}
construct the group $W(\mathbf{E}_{7})$ viewed as a permutation group on $56$
letters, corresponding to the $56$ vertices of the graph $G_2$.
We now enumerate the conjugacy classes of subgroups of $W(\mathbf{E}_{7})$ on using the 
command
\begin{verbatim}
	SubgroupClasses(W_E7);
\end{verbatim}
and find those subgroups $\Gamma \subseteq W(\mathbf{E}_{7})$ which satisfy the conditions of 
Proposition~\ref{prop:criterion}, using standard commands in \texttt{Magma} for computing with permutation groups.

We construct the edges of $G_2$ by constructing the ``intersection matrix'' of $G_2$.
This is defined in a similar manner to the usual adjacency matrix of the graph, except that we define the intersection
of a vertex with itself to be ${(-1)}$, rather than  $0$, as these
correspond to $({-1})$-curves. We first consider the action of $W(\mathbf{E}_{7})$
on unordered pairs of distinct elements of the set $\{1,\ldots, 56\}$, via the command
\begin{verbatim}
	subs2 := Subsets(Set(GSet(W_E7)), 2);
\end{verbatim}
The action of $W(\mathbf{E}_{7})$ 
on \texttt{subs2} decomposes into $3$ orbits, corresponding to 
pairs of vertices with edges between them of multiplicities $0,1$ and $2$.
The orbits have sizes $28,756$ and $756$.
The smallest orbit corresponds to those pairs of vertices which have an edge between them of multiplicity $2$.
To identify which orbit of size $756$ corresponds to the adjacent 
vertices, one considers both possible options in constructing the 
intersection matrix of $G_2$.
One then notices that only one possibility leads to an intersection matrix with the required
rank $8$, the other giving a matrix of rank $29$.

To construct the graphs $G_3$ and $G_1$, one applies a similar method,
but using the root systems $W(\mathbf{E}_{6})$ and $W(\mathbf{E}_{8})$, respectively.
The \texttt{StandardActionGroup} command in \texttt{Magma} does not give the required group actions
for $G_d$ when $d \geq 4$. However these groups, together with
their required actions, are part of \texttt{Magma}'s library of transitive groups.
The commands
\begin{verbatim}
	W_E3 := TransitiveGroup(6,3); 
	W_E4 := TransitiveGroup(10,13); 
	W_E5 := TransitiveGroup(16,1328); 
\end{verbatim}
construct the corresponding groups for $G_6,G_5$ and $G_4$ respectively
(the case $G_7$ in the proof being trivial).
These programs can be found on the authors' webpages.

\begin{remark}
	The part of the algorithm which dominates the run time is the step
	which involves enumerating the conjugacy classes of subgroups of $W(\mathbf{E}_{9-d})$.
	For $d \geq 2$ this takes only a matter of seconds on a modern desktop computer,
	however the case $d=1$ takes around	$30$ minutes.
\end{remark}
We now proceed with the proof of Theorem \ref{thm:main}, following this strategy.

\subsection{Degree $9$}
Let $S$ be a del Pezzo surface of degree $9$.
By Proposition \ref{prop:Hilbert}, the scheme $\mathcal{L}(S)$
is also a del Pezzo surface of degree $9$. Therefore, by a classical 
theorem of Ch\^{a}telet (see e.g.~\cite[Cor.~2.6]{CM96}), the scheme
$\mathcal{L}(S)$ satisfies the Hasse principle, as required.

\subsection{Degree $8$} \label{sec:DP8}
If $S_{\bar k} \not \cong \PP^1 \times \PP^1$
then $S_{\bar k}$ contains only a single line, hence $S$ contains a line
over $k$. Thus the Hasse principle for lines trivially holds.

Let now $S$ be a twist of $\PP^1 \times \PP^1$. Fix a choice of isomorphism
$\Pic S_{\bar k}~\cong~\ZZ^2$, to obtain an injective map
$\Pic S \to \ZZ^2$. We say that a divisor $D$ on $S$ has \emph{type} $(a,b)$,
if its class has image $(a,b)$ under this map. The lines on $S$ are
exactly the effective divisors of type $(1,0)$ or $(0,1)$.

\begin{lemma} \label{lem:C_1xC_2}
	Let $S$ be a twist of $\PP^1 \times \PP^1$ over a number field $k$.
	Suppose that $S$ contains a line over all but finitely many completions of $k$.
	Then $\mathcal{L}(S)$ is reducible. In particular, 
	there exist conics $C_1$ and $C_2$	over $k$, unique up to isomorphism and reordering,
	such that
	$$S \cong C_1 \times C_2 \quad \mbox{and} \quad \mathcal{L}(S) \cong C_1 \sqcup C_2.$$
\end{lemma}
\begin{proof}
	By Proposition \ref{prop:Hilbert}, the scheme $\pi_0(\mathcal{L}(S))$ (see \S \ref{sec:notation})
	is a finite \'{e}tale scheme of degree $2$.
	By assumption, it has a point over all but finitely many completions of $k$.
	It follows easily from Lemma \ref{lem:Chebotarev} that $\pi_0(\mathcal{L}(S))$
	is split. Hence, by Proposition \ref{prop:Hilbert} there exist conics $C_1$ and $C_2$
	such that 
	$$
		\mathcal{L}(S) \cong C_1 \sqcup C_2. \label{eqn:two_conics}
	$$
	In particular, the surface
	$S$ contains divisors of type $(2,0)$ and $(0,2)$. The morphisms
	associated to these divisors give maps $S \to C_i$,
	which when combined together give the required isomorphism $S \cong C_1 \times C_2$.
\end{proof}

It is quite easy to construct counter-examples here.
Namely, by the classification of conics over $k$
(see e.g \cite[Thm.~72.1]{O'Me73}),
there exist two conics $C_1$ and $C_2$ over 
$k$ with $C_1(k) = \emptyset$ and $C_2(k) = \emptyset$,
yet for each place $v$ of $k$ there is some $i \in \{1,2\}$
such that 
$C_i(k_v) \neq \emptyset$.
In which case, the surface
$$S=C_1 \times C_2,$$
is a counter-example to the Hasse principle for lines.
Indeed, the scheme $\mathcal{L}(S) = C_1 \sqcup C_2$ fails the Hasse principle.
On the other hand, we have the following.

\begin{lemma} \label{lem:quadrics}
	Let $S \subset \PP^3$ be a smooth quadric surface over a number field $k$.
	Then $S$ satisfies the Hasse principle for lines.
\end{lemma}
\begin{proof}
	Assume that $S$ contains a line over every completion of $k$.
	Then $S$ is everywhere locally soluble, hence
	$\Pic S \cong (\Pic S_{\bar k})^{G_k}$ (see e.g. \cite[Cor.~2.5]{CM96}).
	By Lemma \ref{lem:C_1xC_2} we know that $S$ contains a divisor
	of type $(2,0)$. Thus any divisor class on $S_{\bar k}$ of type $(1,0)$ is Galois
	invariant, hence is defined over~$k$. Therefore $S$ contains a line,
	so the Hasse principle for lines holds.
\end{proof}

\subsection{Degree $7$}
The graph $G_7$ is the path graph on $3$ vertices.
In particular the automorphism group of the graph fixes a vertex,
so there is always a line defined over $k$. Hence the Hasse principle
for lines trivially holds.

\subsection{Degree $6$}
Here the graph $G_6$ is the cycle graph on $6$ vertices, which
we identify with a regular hexagon. The automorphism group is the dihedral group $D_6$,
acting in the usual way.

Let $\Gamma \subseteq D_6$ be a subgroup. 
If $\Gamma$ is transitive or cyclic, then 
by Lemma \ref{lem:HP}
the Hasse principle holds. The group $D_6$
contains only two conjugacy classes of non-transitive non-cyclic subgroups,
which are isomorphic to the Klein four group $V_4$
and the dihedral group $D_3$ respectively.
Both these groups admit an element which acts
as a non-trivial rotation, hence has no fixed point. Thus by Proposition 
\ref{prop:criterion}, the Hasse principle for lines holds.

\subsection{Degree $5$} \label{sec:DP5}
\subsubsection{Results of computer experiments}
We now present the results obtained from our computer
experiments, following our proof strategy. Here
$G_5$ is the Petersen graph, with automorphism group
$W(\mathbf{E}_4) \cong W(\mathbf{A}_4) \cong S_5$ of order $120$. Our results are summarised
in Table \ref{tab:DP5}.

\begin{table}[htb]
 \centering
 \begin{tabular}{c|c|c|c|c}
  subgroups & transitive & cyclic & fixed point & satisfy Prop.~\ref{prop:criterion} \\
  \hline
  $19$&$3$&$7$&$9$ &$2$
 \end{tabular} 
 \caption{Results of experiments for del Pezzo surfaces of degree $5$}
 \label{tab:DP5}
\end{table}
We now explain the notation in the table, with similar notation occurring
for the other tables in our paper. The first column gives the 
number of \emph{conjugacy classes} of subgroups of $W(\mathbf{E}_4)$. We only consider
subgroups up to conjugacy, as conjugate actions give rise to isomorphic finite \'{e}tale
schemes. We then list the number of these which act transitively on $G_5$, the number 
which are cyclic, the number which act with a fixed
point, and the number which satisfy the conditions of Proposition \ref{prop:criterion}.
Ultimately only the last column will be of interest, but we have
included the other columns to give a more complete picture.

The two subgroups of $W(\mathbf{E}_4)$ which satisfy Proposition \ref{prop:criterion}
are isomorphic to $V_4$ and $A_4$, respectively.
They have the following orbit types.
\begin{enumerate}
	\item $V_4$ with orbit type $[2,2,2,4]$.
	\item $A_4$ with orbit type $[4,6]$.
\end{enumerate}


In both cases, the orbits of length $4$ consist of mutually non-adjacent vertices.
This means that the corresponding del Pezzo surfaces (if they exist)
contain $4$ skew lines. Contracting these lines, we obtain a del Pezzo surface $S'$
of degree $9$. A theorem of Enriques and Swinnerton-Dyer (see e.g.~\cite[Cor.~3.1.5]{Sko01})
states that any del Pezzo surface of degree $5$ contains a rational point, in particular
so does $S'$ and hence $S' \cong \PP^2$. Thus any such surface 
is obtained by blowing-up $\PP^2$ in a closed point
of degree $4$.  We now explain how to realise these orbit types
and how to construct counter-examples in each case.

\subsubsection{$V_4$} \label{sec:V_4}
Let $K/k$ be a $V_4$-extension and let $S$ be the blow-up
of $\PP^2$ in a closed point of degree $4$ in general position, with residue field
$K$. Using the description of the lines on $S$ given in \S \ref{sec:classification},
it is simple to see that $S$ has the required orbit type $[2,2,2,4]$.
In particular, there are three non-conjugate subgroups of $V_4$ of order $2$
which act with a fixed point on $G_5$.
To see that $S$ satisfies the conditions of Proposition \ref{prop:criterion},
it suffices to note the following:
\begin{itemize}
	\item $V_4$ contains exactly three subgroups of order $2$.
	\item Every element of $V_4$ lies inside such a subgroup.
\end{itemize}
Therefore $S$ contains a line over almost all completions of $k$.
To take care of the ramified primes, we may choose $K/k$ as in Lemma \ref{lem:solvable},
which guarantees that $S$ contains lines over \emph{all} completions of $k$.
Thus $S$ is the required counter-example to the Hasse principle for lines.

As an example, one may blow-up
the closed point which over $\bar k$ is the union of
$$[\sqrt{a}:\sqrt{b}:1],[-\sqrt{a}:\sqrt{b}:1], [\sqrt{a}:-\sqrt{b}:1] \mbox{ and } [-\sqrt{a}:-\sqrt{b}:1], $$
where $a,b$ are as in Lemma \ref{lem:biquadratic}. 

\subsubsection{$A_4$} \label{sec:A_4}
Here the construction is similar to the previous case.
One takes $S$ to be the blow-up of $\PP^2$ in a closed point
of degree $4$, whose splitting field $K$ has Galois group $A_4$.
The action of $A_4$ is transitive on the collection of lines passing
through these $4$ points over $\bar k$, thus the orbit type is the required $[4,6]$.
In particular there are subgroups of $A_4$ of orders $2$ and $3$ which act 
with a fixed point on $G_5$.
To see that $S$ satisfies the conditions of Proposition \ref{prop:criterion},
it suffices to note the following:
\begin{itemize}
	\item $A_4$ contains only one conjugacy class of subgroups of orders $2$ and $3$.
	\item Every element of $A_4$ lies inside such a subgroup.
\end{itemize}
Thus on choosing $K$ as in Lemma
\ref{lem:solvable}, we obtain the required counter-example.

As an explicit example of $K$ which works when $k=\QQ$,
one may take the splitting field of the
polynomial
$$f(x)=x^4-x^3-7x^2+2x+9.$$
This polynomial was taken from \cite{KM01}.
Let $L$ be the field defined by $f$. It has discriminant $163^{2}$
and its Galois closure $K$ has Galois group $A_4$.
It follows from Lemma \ref{lem:Hasse}
that this gives the required counter-example, 
on noticing that we have a factorisation
$$(163) = \mathfrak{p}_1 \mathfrak{p}_2^3,$$
into prime ideals of $L$, where both $\mathfrak{p}_1$
and $\mathfrak{p}_2$ have inertia degree $1$.

\subsection{Degree $4$}
\subsubsection{Results of computer experiments}
Here $G_4$ is the Clebsch graph, with automorphism group
$W(\mathbf{E}_5) \cong W(\Dfive) \cong (\ZZ/2\ZZ)^4 \rtimes S_5$ of order $1920$.
Our results are summarised in Table \ref{tab:DP4}.

\begin{table}[ht]
 \centering
 \begin{tabular}{c|c|c|c|c}
  subgroups & transitive & cyclic & fixed point & satisfy Prop.~\ref{prop:criterion} \\
  \hline
  $197$&$51$&$18$&$19$&$0$
 \end{tabular} 
 \caption{Results of experiments for del Pezzo surfaces of degree $4$}
 \label{tab:DP4}
\end{table}
There are $116$ conjugacy classes of subgroups of $W(\mathbf{E}_5)$
which are not transitive, nor cyclic, nor admit a fixed point.
Moreover, none of these satisfy the conditions of Proposition \ref{prop:criterion}.
Hence the Hasse principle for lines holds.
We will give a non-computer-assisted proof 
of this result in \S \ref{sec:conceptual}.

\subsection{Degree $3$} \label{sec:DP3}
\subsubsection{Results of computer experiments}
Here $G_3$ is the (dual of the) Schl\"{a}fli graph,
with automorphism group $W(\mathbf{E}_6)$ of order $51840$.
Our results are summarised in Table \ref{tab:DP3}.

\begin{table}[ht]
 \centering
 \begin{tabular}{c|c|c|c|c}
  subgroups & transitive & cyclic & fixed point & satisfy Prop.~\ref{prop:criterion} \\
  \hline
  $350$&$25$&$25$&$172$&$3$
 \end{tabular}
 \caption{Results of experiments for del Pezzo surfaces of degree $3$}
 \label{tab:DP3}
\end{table}
Our experiments show that exactly $3$ out of the 
$350$ conjugacy classes of $W(\Esix)$ satisfy
the conditions of Proposition \ref{prop:criterion}.
Surprisingly, these all admit a very special orbit type.
The three subgroups which arise are as follows.
\begin{enumerate}
	\item $D_5$ with orbit type $[2,5,5,5,10]$.
	\item $\ZZ/5\ZZ \rtimes \ZZ/4\ZZ$ with orbit type $[2,5,10,10]$.
	\item $S_5$ with orbit type $[2,5,10,10]$.
\end{enumerate}
Here $\ZZ/5\ZZ \rtimes \ZZ/4\ZZ$ denotes a semi-direct
product of $\ZZ/5\ZZ$ with $\ZZ/4\ZZ$ for which the induced
morphism $\ZZ/4\ZZ \to (\ZZ/5\ZZ)^*$ is an isomorphism
(such a group is unique up to isomorphism).

In each case, the orbit of length two consists of two non-adjacent 
vertices in the graph, and there is also an orbit of length five
consisting of mutually non-adjacent vertices. Therefore any cubic surface
of this type, if it exists, admits two skew lines defined over a 
quadratic field extension. Such surfaces are rational. Indeed, 
contracting the two skew conjugate lines
we obtain a del Pezzo surface of degree $5$, and any such surface is 
rational \cite[Thm.~29.4]{Man86}.

\subsubsection{Constructing these surfaces} \label{sec:constructions}

We now explain how to construct these surfaces. Choose two closed points $P$ and $Q$ in $\PP^2$
of degrees $2$ and $5$, respectively, such that $P \sqcup Q$ is in general position.
The del Pezzo surface $\smash{S' = \Bl_{P \sqcup Q} \PP^2}$ of degree $2$ contains exactly two lines.
These are the strict transforms $\smash{\widetilde{L}}$
and $\smash{\widetilde{C}}$ of the line $L$ through the two quadratic points and the conic $C$
passing through the five quintic points over $\smash{\bar k}$.
Contracting $\smash{\widetilde{L}}$, we obtain a cubic surface $S$
which contains no lines, since $\smash{\widetilde{C}}$ and $\smash{\widetilde{L}}$ intersect.
Note that if we chose instead to contract 
$\smash{\widetilde{C}}$, we would obtain an isomorphic surface,
as $\smash{\widetilde{L}}$ and $\smash{\widetilde{C}}$ are exchanged by the
Geiser involution on $\smash{S'}$ (see \cite[\S 8.7.2]{Dol12}). One way to see this is by
noting that the Geiser involution preserves the intersection product and
fixes no line, combined with the fact that $\smash{\widetilde{C}}$ is the unique line on $S'$ with 
$\smash{\widetilde{C} \cdot \widetilde{L} =2}$. Using the classification given
in \S \ref{sec:classification}, the lines on $S$ correspond to the following curves:
\begin{enumerate}
 \item The $2$ singular cubic curves which pass through all $7$ points and which have a double point
 at exactly one of the quadratic points.
 \item The $5$ exceptional curves above the quintic points.
 \item The $10$ conics passing through the two quadratic points and three of the five quintic points.
 \label{item:conics}
 \item The $10$ lines passing through one of the quadratic points and one of the quintic points.
 \label{item:lines}
\end{enumerate}

This description gives a simple geometric way to visualise the Galois 
action on the lines. Such surfaces always contain an orbit of length two and an orbit
of length five, which each consist of skew lines, together with an orbit of length ten
given by curves of the type $\eqref{item:lines}$. The union of this latter orbit
with the orbit of length two is a Galois invariant double six on the surface (see \cite[Rem.~V.4.9.]{Har77}
for a definition of this configuration). The Galois action on the curves occurring in 
$\eqref{item:conics}$ is the same as the action of the Galois group of the splitting field $K$ of 
$Q$ on pairs of distinct $K$-points of~$Q$. Of the five transitive conjugacy
classes of subgroups $\Gamma$ of $S_5$, there are two orbits of length five
exactly when $\Gamma \cong \ZZ/5\ZZ$ or $\Gamma \cong D_5$.
Hence this constructs surfaces with orbit types $[2,5,5,5,10]$ and $[2,5,10,10]$. Moreover,
it is not difficult to see that every such surface arises in this way.

We now specialise the residue fields of $P$ and $Q$ to obtain the
cases which arose from our computer experiments. In what follows we keep the above notation.

\subsubsection{$D_5$} \label{sec:D_5}
Assume now that the residue fields of $P$ and $Q$ belong to the same $D_5$-extension
$K/k$. The orbit type is $[2,5,5,5,10]$ in this case.
Thus there are subgroups of $D_5$ with orders $2$ and $5$ which act with a fixed point on $G_3$.
To see that $S$ satisfies the conditions of Proposition \ref{prop:criterion},
it suffices to note the following:
\begin{itemize}
	\item $D_5$ contains only one conjugacy class of subgroups of orders $2$ and $5$.
	\item Every element of $D_5$ lies inside such a subgroup.
\end{itemize}
Therefore we see that $S$ contains a line over almost all completions of $k$.
To obtain the required counter-example over $k$, we choose $K$ as
in Lemma \ref{lem:solvable}.

For an explicit example of such a $K$ when $k=\QQ$,
one may take the splitting field of the
polynomial
$$f(x)=x^5 - 2x^4 + 2x^3 - x^2 + 1.$$
Using \texttt{Magma}, one finds that the field $L$ defined by $f$ has
discriminant $47^{2}$, and its splitting field $K$
is a $D_5$-extension of $\QQ$. 
It follows from Lemma \ref{lem:Hasse}
that this gives the required counter-example, 
on noticing that we have a factorisation
$$(47) = \mathfrak{p}_1 \mathfrak{p}_2^2 \mathfrak{p}_3^2,$$
into prime ideals of $L$, where each prime
ideal has inertia degree $1$.

\subsubsection{$\ZZ/5\ZZ \rtimes \ZZ/4\ZZ$} \label{sec:Z/5ZxZ/4Z}
Here we let $\Gamma = \ZZ/5\ZZ \rtimes \ZZ/4\ZZ$.
We assume that the splitting field $K/k$ of $Q$ 
is a $\Gamma$-extension of $k$,
and that the residue field of $P$ is the unique quadratic subfield
of $K$. In which case the orbit type is $[2,5,10,10]$, hence there are subgroups of $\Gamma$ 
of orders $4$ and $10$ which act with a fixed point on~$G_3$.
To see that $S$ satisfies the conditions of Proposition \ref{prop:criterion},
it suffices to note the following:
\begin{itemize}
	\item $\Gamma$ contains only one conjugacy class of subgroups of order $4$ and $10$.
	\item Every element of $\Gamma$ lies inside such a subgroup.
\end{itemize}
As before, we obtain a counter-example on choosing $K$ as
in Lemma \ref{lem:solvable}.

Here one can write down a simple family of such fields when $k=\QQ$.
Namely, let $p$ be a prime which is congruent to $1$ modulo $25$ 
(infinitely such primes exist by Dirichlet's theorem).
We then take $K$ to be the splitting field of the polynomial
$$f(x)=x^5 - p.$$
It is simple to check that this is a $\Gamma$-extension
and that the unique quadratic subfield of $K$ is $\QQ(\sqrt{5})$. 
Therefore $\mathcal{L}(S)$ contains a subscheme isomorphic to
$$(x^2 - 5)(x^5 - p) = 0.$$
It suffices to show that this scheme fails the Hasse principle. It satisfies
the conditions of Lemma \ref{lem:Chebotarev}, hence we need only look at the ramified primes, namely $5$ and $p$.
As $p \equiv 1 \pmod 5$, the field $\QQ(\sqrt{5})$ splits at
$p$, thus $S$ contains a line over $\QQ_p$. Moreover, the polynomial
$f$ has a non-singular root modulo $25$, hence by Hensel's lemma it admits a root over $\QQ_5$.
We therefore deduce that $S$ fails the Hasse principle for lines, as required.

The explicit cubic surface \eqref{eqn:cubic} was of this type, with $p=101$.
We shall explain how we wrote down the equation in this case in \S \ref{sec:explicit}.

\subsubsection{$S_5$} \label{sec:S_5}
We now assume that the splitting field $K/k$ of $Q$ 
is an $S_5$-extension of $k$,
and that the residue field of $P$ is the unique quadratic subfield
of $K$. The orbit type here is $[2,5,10,10]$, hence there are subgroups of $S_5$
of orders $12,12,24$ and $60$ which act with a fixed point on $G_3$.
Using the explicit description of the lines on $S$ given in \S \ref{sec:constructions},
one checks that the two subgroups of order $12$ are non-conjugate in $S_5$.
To see that $S$ satisfies the conditions of Proposition \ref{prop:criterion},
it suffices to note the following:
\begin{itemize}
	\item Up to conjugacy, $S_5$ contains exactly two subgroups of order $12$,
	one of order $24$ and one of order $60$.
	\item Every element of $S_5$ lies inside such a subgroup.
\end{itemize}
Therefore $S$ contains a line over almost all completions of $k$.
To handle the ramified primes we will have to choose $K$ very carefully.
Note that Lemma \ref{lem:solvable} does not apply as $S_5$ is non-solvable.

We begin by noticing that the only solvable subgroups	of $S_5$ which do not act
with a fixed point on $G_3$ are all isomorphic to $\ZZ/5\ZZ \rtimes \ZZ/4\ZZ$ 
(in fact every subgroup of $S_5$ of order $20$ has this form). We need 
only consider solvable subgroups, since decomposition groups are always solvable 
\cite[Cor.~IV.2.5]{Ser80}. We shall therefore choose $K$ so that such groups do not
occur as decomposition groups.

\begin{lemma} \label{lem:S_5}
	Let $k$ be a number field. Then there exists an $S_5$-extension
	$K/k$, none of whose decomposition groups is isomorphic to 
	$\ZZ/5\ZZ \rtimes \ZZ/4\ZZ$.
\end{lemma}
\begin{proof}
	We shall prove the result by constructing suitably infinitely 
	many such fields over $\QQ$ using \cite{Ked12}, then take a compositum with $k$.
	
	Let $F/\QQ$ be a quintic field with square-free discriminant,
	whose Galois closure $K$ has Galois group $S_5$. If $N$ denotes the unique quadratic subfield of $K$,
	then the extension $K/N$ is unramified (see  \cite{Ked12}).
	It follows that if $\mathfrak{p}$ is a ramified prime ideal of $K/\QQ$, then 
	$\mathfrak{p}$ has ramification index $2$. Hence the decomposition group of
	$\mathfrak{p}$ has a normal subgroup of order $2$. However $\ZZ/5\ZZ \rtimes \ZZ/4\ZZ$
	contains no such subgroup, thus $K$ satisfies the conditions of the lemma.
	
	We now handle the case of a general number field $k$. 
	Choose a quintic field $F/\QQ$ as above such that the discriminant $\Delta_{F/\QQ}$ 
	is coprime to $\Delta_{k/\QQ}$. Such fields exist by \cite[Cor.~1.3]{Ked12}.
	Let $L$ be the Galois closure of $F$ and let $N$ be the unique
	quadratic subfield of $L$.	On embedding $k$ and $L$ in a common algebraic closure, we 
	obtain the following commutative diagram of field extensions
	$$
	\xymatrix{k \ar[r]&	Nk  \ar[r]& Lk\\
	\QQ \ar[u] \ar[r] & N  \ar[r] \ar[u]& L . \ar[u]} 
	$$
	In order to prove the lemma with $K=Lk$, we may apply a similar argument as above 
	to reduce to showing that 
	\begin{enumerate}
		\item $k \subset Lk$ is an $S_5$-extension. \label{item:S_5}
		\item $Nk \subset Lk$ is unramified. \label{item:unramified}
	\end{enumerate}
	The assumption $(\Delta_{F/\QQ},\Delta_{k/\QQ})=1$ implies \cite[Thm.~4.9]{Nar04} that
	$L$ and $k$ are linearly disjoint over $\QQ$. Here \eqref{item:S_5}
	is then standard \cite[Thm.~VI.1.12]{Lan02}.
	Also as $N \subset L$ is unramified, \eqref{item:unramified} follows from 
	\cite[Prop.~7.2]{Neu99}, as required.
\end{proof}

Therefore choosing $K/k$ as in Lemma \ref{lem:S_5} and using Lemma \ref{lem:decomposition},
we see that the surface $S$ also admits a line over the completion of $k$ at every place which is ramified in $K$.
Hence we obtain the required counter-example over $k$.

As an explicit choice over $\QQ$, take $K$ to be the splitting field
of the polynomial
$$f(x)=x^5-x^4-5x^3+5x^2+2x-1.$$
This polynomial was taken from \cite{KM01}.
The discriminant of the field defined by $f$ is 
equal to $101833$, which is prime, and the Galois group
of $K$ is $S_5$.
As the discriminant is square-free,
the proof of Lemma \ref{lem:S_5} shows that it satisfies
the required properties (in fact the 
decomposition group of each ramified prime ideal
of $K$ is isomorphic to $\ZZ/2\ZZ$ here).

\subsubsection{An explicit counter-example} \label{sec:explicit}
We now explain how we were able to write down the equation for the 
counter-example \eqref{eqn:cubic} given in the introduction. To do this, we shall use the constructions
given in \S \ref{sec:constructions} and \S \ref{sec:Z/5ZxZ/4Z}.

Recall that we choose a point in $\PP^2$ defined over $\QQ(\sqrt[5]{101})$ and a point defined over
$\QQ(\sqrt{5})$ such that the associated orbits form seven points in general position,
then blow them up to obtain an auxiliary del Pezzo surface $S'$ of degree $2$. The surface
$S'$ contains two exceptional curves that are $\QQ$-rational, and blowing-down
one of them we find a cubic surface $S$ with the desired properties.

To do this algorithmically, we determine three linearly independent cubic forms
$c_0, c_1, c_2$ on $\PP^2$ that vanish at the seven points. These define the 
anticanonical map $S'\to\PP^2$. We choose $c_0$
to be the product of a linear form vanishing on the size-two orbit and a quadratic
form vanishing on the size-five orbit.
We also determine a sextic form $s$ which admits a double point at each of the seven blow-up points,
and is not a linear combination of products of the cubic forms found. 
We then calculate the unique relation of the type
$$s^2 + f_2(c_0,c_1,c_2)s + f_4(c_0,c_1,c_2) = 0$$
between these forms. Here $f_2$ denotes a ternary quadratic form, while
$f_4$ denotes a ternary quartic form.

Up to now, these computations are pure linear algebra. They lead to a del Pezzo surface 
of degree $2$ with the explicit equation
$$w^2 + f_2(x_0,x_1,x_2)w + f_4(x_0,x_1,x_2) = 0 \quad \subset \PP(1,1,1,2).$$
By our choice of $c_0$, the pull-back of the line $x_0=0$, via the natural map
$S' \to \PP^2$, decomposes into a union of two lines on $S'$.
Thus a linear transformation yields an equation of the form
$$w^2 + f'_2(x_0,x_1,x_2)w + x_0 \cdot f_3(x_0,x_1,x_2) = 0 ,$$
where $f'_2$ is a ternary quadratic form and $f_3$ is a ternary cubic form.
It is now classically known how to explicitly contract one the exceptional curves lying over the line
$x_0 = 0$. As was shown by C. F. Geiser \cite[\S III]{Gei69} already in $1869$, the result is the 
cubic surface $S$ given by the equation
$$x_0 \cdot w^2 + f'_2(x_0,x_1,x_2)w + f_3(x_0,x_1,x_2) = 0 \quad \subset \PP^3 .$$
This is the equation of the required counter-example.
To obtain small coefficients as in \eqref{eqn:cubic}, one may apply
Kollar's minimisation algorithm \cite[Prop.~6.4.2]{Kol97}.

\begin{remark}
	i) Explicit examples for the cases treated in \S \ref{sec:D_5} and \S \ref{sec:S_5}
	may be obtained in a similar manner. 

	\noindent ii) As we have seen, surfaces with orbit type $[2,5,5,5,10]$ or $[2,5,10,10]$ 
	have a Galois invariant double-six. Thus, the methods developed in \cite{EJ10} apply 
	and lead to explicit equations for each of the three groups over $\QQ$.
	Our approach here is simpler however, since it is more direct and easier to 
	implement (for example, we do not need to perform any explicit Galois descent, as in \cite{EJ10}).
\end{remark}

\subsection{Degree $2$} \label{sec:DP2}
\subsubsection{Results of computer experiments}
Here the group $\Aut(G_2)= W(\mathbf{E}_7)$ has order $2903040$.
Our results are summarised in Table \ref{tab:DP2}.

\begin{table}[ht]
 \centering
 \begin{tabular}{c|c|c|c|c}
  subgroups & transitive & cyclic & fixed point & satisfy Prop.~\ref{prop:criterion} \\
  \hline
  $8074$&$32$&$60$&$350$&$60$
 \end{tabular}
 \caption{Results of experiments for del Pezzo surfaces of degree $2$}
 \label{tab:DP2}
\end{table}
Here a much wider range of behaviour occurs than in higher degrees.
We therefore construct counter-examples in some special
cases to illustrate this range, including some surfaces
which are minimal.

\begin{remark}
	Of the $60$ conjugacy classes of subgroups which arise, exactly $2$ 
	are non-solvable, being isomorphic to $A_5$ and $A_5 \times \ZZ/2\ZZ$, respectively.
	Moreover surfaces with these Galois actions are rational, being blow-ups of $\PP^1 \times \PP^1$
	in a closed point of degree six. It seems likely that the method used for cubic surfaces
	(\S \ref{sec:constructions} and Lemma \ref{lem:S_5})
	can be adapted to construct counter-examples in such cases.
	Also, of these $60$ conjugacy classes of subgroups, exactly $30$ correspond
	to minimal surfaces. We shall not pursue all these additional cases in this paper.
\end{remark}

\subsubsection{Some rational counter-examples}

\begin{lemma} \label{lem:DP2_blow-up}
	Let $P,Q$ be closed points of $\PP^2$ of degrees $3$ and $4$,
	respectively, such that $P \sqcup Q$ lies in general position. 
	Then $\Bl_{P \sqcup Q} \PP^2$ contains no line.
\end{lemma}
\begin{proof}
	It suffices to show that none of the curves in $\PP^2$ described in \S \ref{sec:classification}
	is Galois invariant.	
	As there is no orbit of length $2$ or $5$, no line nor conic is fixed by the Galois action.
	No cubic curve passing through the seven points 
	which admits a double point at exactly one of them is Galois invariant,
	as there is no rational point being blown-up.
	This completes the proof.
\end{proof}

\begin{lemma} \label{lem:DP2_counter-example}
	Let $P,Q$ be closed points of $\PP^2$ of degrees $3$ and $4$,
	respectively, such that $P \sqcup Q$ lies in general position. Assume
	that the splitting field $K$ of $Q$ is a $V_4$- or $A_4$-extension of $k$
	which satisfies the conditions of Lemma \ref{lem:solvable}.
	Then $\Bl_{P \sqcup Q} \PP^2$ is a del Pezzo surface of degree $2$
	which is a counter-example to the Hasse principle for lines.
\end{lemma}
\begin{proof}
	It follows from Lemma \ref{lem:DP2_blow-up} that $S=\Bl_{P \sqcup Q} \PP^2$ contains no line.
	However $S$ is the blow-up in a closed point of degree $3$
	of a del Pezzo surface of degree $5$ which is a counter-example 
	to the Hasse principle for lines (see \S \ref{sec:V_4} and \S \ref{sec:A_4}).
	In particular $S$ itself is a counter-example to the Hasse principle for lines.
\end{proof}

\begin{remark}
	One can check that the construction given in Lemma \ref{lem:DP2_counter-example}
	yields $11$ out of the $60$ subgroups of $W(\mathbf{E}_7)$ which satisfy 
	Proposition \ref{prop:criterion}. Generically we obtain surfaces with 
	orbit type
	$$[3,3,3,3,4,4,6,6,12,12],$$
	and all other surfaces thus obtained have refinements of this orbit type.
\end{remark}

\subsubsection{A minimal counter-example}
We now give a minimal surface which is a counter-example to the Hasse
principle for lines. We do this by constructing del Pezzo surfaces
of degree $2$ equipped with a conic bundle structure, using the methods of 
\cite[\S 5]{BMS14}. Consider a conic bundle surface of the shape
$$f(t) x^2 + g(t)y^2 + h(t) z^2=0 \quad \subset \mathbb{A}^1 \times \PP^2,$$
where 
$$ f(t) = a(t-13)(2-t), \quad g(t) = b(t + 14)(3-t), \quad h(t) = (t+2)(t-11),$$
and $a,b \in k^*$. By \cite[Prop.~5.2]{BMS14}, the closure of this in $\PP^1 \times \PP^2$
is a del Pezzo surface $S$ of degree $2$.
The natural projection realises $S$ as a conic bundle over $\PP^1$,
with exactly $6$ singular fibres. If $a,b$ and $ab$ are not squares in $k$,
then one checks that each of these singular fibres
is non-split, and that exactly $2$ fibres become split over each of one of the field extensions
$$k(\sqrt{a}), k(\sqrt{b}), k(\sqrt{ab}).$$
Thus this conic bundle is relatively
minimal, hence minimal by a result of Iskovskih \cite[Thm.~4]{Isk79}.
In particular the surface $S$ contains no line.
Another result of Iskovskih implies that such surfaces are non-rational
\cite[Cor.~1.7]{Isk70}. If we choose $a,b$ as in Lemma \ref{lem:biquadratic},
then for each place $v$ of $k$, at least one of the singular fibres
splits into a union of lines. 
In particular, the surface contains a line over every completion
of $k$, thus yields a counter-example to the Hasse principle for lines.
Such surfaces are split over $k(\sqrt{a},\sqrt{b})$, and 
have orbit type $$[2,2,2,2,2,2,2,2,2,2,2,2,4,4,4,4,4,4,4,4].$$

\subsection{Degree $1$} \label{sec:DP1}
\subsubsection{Results of computer experiments}
Here the group $\Aut(G_1)= W(\mathbf{E}_8)$ has order $696729600$.
Our results are summarised in Table \ref{tab:DP1}.

\begin{table}[hbt]
 \centering
 \begin{tabular}{c|c|c|c|c}
  subgroups & transitive & cyclic & fixed point & satisfy Prop.~\ref{prop:criterion} \\
  \hline
  $62092$&$60$&$112$&$7735$&$8742$
 \end{tabular}
 \caption{Results of experiments for del Pezzo surfaces of degree $1$}
 \label{tab:DP1}
\end{table}

The number of possible subgroups is so large that it is not feasible with
current technology and methods to consider all cases. Therefore, as before,
we simply construct counter-examples in some special cases.
We remark that out of the $8742$ conjugacy classes of subgroups
of $W(\mathbf{E}_8)$ which satisfy the conditions of Proposition~\ref{prop:criterion},
exactly $7775$ are solvable.

\subsubsection{Some rational counter-examples}

\begin{lemma} \label{lem:DP1_blow-up}
	Let $P,Q$ be closed points of $\PP^2$ of degree $4$
	such that $P \sqcup Q$ lies in general position. 
	Then $\Bl_{P \sqcup Q} \PP^2$ contains no line.
\end{lemma}
\begin{proof}
	As in Lemma \ref{lem:DP2_blow-up}, an inspection of 
	the curves in $\PP^2$ given in \S \ref{sec:classification} reveals
	that none of them is Galois invariant, hence none are defined over $k$.
\end{proof}

\begin{lemma} \label{lem:DP1_counter-example}
	Let $P,Q$ be closed points of $\PP^2$ both of degree $4$,
	such that $P \sqcup Q$ lies in general position. Assume
	that the splitting field $K$ of $Q$ is a $V_4$- or $A_4$-extension of $k$
	which satisfies the conditions of Lemma \ref{lem:solvable}.
	Then $\Bl_{P \sqcup Q} \PP^2$ is a del Pezzo surface of degree $1$
	which is a counter-example to the Hasse principle for lines.
\end{lemma}
\begin{proof}
	As in the proof of Lemma \ref{lem:DP2_counter-example},
	by Lemma \ref{lem:DP1_blow-up} we know that $S$ contains no line,
	yet is a blow-up of a del Pezzo surface of degree $5$
	which is a counter-example to the Hasse principle for lines.
	The result follows.
\end{proof}

\begin{remark}
	The construction given in Lemma \ref{lem:DP1_counter-example}
	yields $26$ out of the $8742$ subgroups of $W(\mathbf{E}_8)$ which satisfy 
	Proposition \ref{prop:criterion}. Generically this yields surfaces with orbit
	type
	$$[4,4,4,4,4,4,4,4,6,6,6,6,12,12,16,16,16,16,24,24,24,24],$$
	and all other surfaces thus obtained have refinements of this orbit type.
\end{remark}

\subsubsection{A minimal counter-example}
We now construct some minimal counter-examples using conic bundles.
We consider conic bundles of the shape
$$z^2 = f(t)x^2 + g(t) y^2 \quad \subset \mathbb{A}^1 \times \PP^2,$$
where $f,g \in k[t]$ and $\deg f = \deg g = 4.$
This is a variant of the family considered in \cite[\S 5]{BMS14}.
Given $e_1,e_2,e_3,e_4,\lambda,a,b \in k^*$, we take $f$ to be
$$f(t)=(t-e_1)(t-e_2)(t-e_3)(t-e_4),$$
and also take 
$$g(t) = \lambda f(t) + \prod_{i\neq1}\frac{t-e_i}{e_1-e_i}
+ a \prod_{i\neq2}\frac{t-e_i}{e_2-e_i}
+ b \prod_{i\neq3}\frac{t-e_i}{e_3-e_i}
+ ab \prod_{i\neq4}\frac{t-e_i}{e_4-e_i}.$$ 
This has been defined so that it satisfies
$$g(e_1) = 1, \quad g(e_2) = a, \quad g(e_3) = b, \quad g(e_4) = ab.$$

With these choices, the fibre over $t=e_1$ is split. Contracting one of the lines in this fibre, 
a similar method to the proof of \cite[Prop.~5.3]{BMS14} shows that we obtain a surface
with a natural
compactification $S$ which is a del Pezzo surface of degree $1$ with a conic bundle structure,
provided that the $e_i$ and $\lambda$ are chosen sufficiently generally. 
This conic bundle has
$7$ singular fibres over $\bar k$. By construction, three of these fibres occur
over rational points, with splitting fields
$$k(\sqrt{a}), k(\sqrt{b}), k(\sqrt{ab}).$$
Moreover if the $e_i$ and $\lambda$ are chosen sufficiently generally, then 
the conic bundle has a non-split fibre over a closed
point of degree $4$. Thus the surface is relatively minimal,
hence minimal and non-rational by results of Iskovskih (\cite[Thm.~4]{Isk79}
and \cite[Cor.~1.7]{Isk70}). If $a,b$ satisfy the conditions
of Lemma \ref{lem:biquadratic}, we see that for each place $v$ of $k$, at least one of the singular fibres
splits into a union of lines. Thus we obtain the required counter-example to the Hasse principle for lines.

\begin{remark}
	Generically, this construction leads to surfaces with orbit type
	$$[2, 2, 2, 2, 2, 2, 4, 4, 4, 8, 8, 16, 16, 16, 24, 32, 32, 32, 32]$$
	whose Galois group is a solvable group of order $768$, and all other surfaces obtained 
	in this way have refinements of this orbit type.
\end{remark}

This completes the proof of Theorem \ref{thm:main}. \qed

\section{Distribution of counter-examples} \label{sec:thin}
The aim of this section is to prove Theorem \ref{thm:thin_cubic}, together with an analogous
result for del Pezzo surfaces. We begin by giving the 
set-up for our results.

\subsection{The set-up}
We work for now over $\ZZ$, and construct
a Hilbert scheme which parametrises del Pezzo surfaces
of degree $d$.  Define
$$X_d  = 
\begin{cases}
	\PP^d, & \quad  \mbox{if } d \geq 3, \\
	\PP(1,1,1,2), & \quad  \mbox{if } d = 2, \\
	\PP(1,1,2,3), & \quad  \mbox{if } d = 1.
\end{cases}
$$

So that any del Pezzo surface of degree $d$ may be embedded anticanonically
into~$X_d$. We define the scheme $\mathcal{H}_d$ to be the Hilbert scheme over $\ZZ$ which 
parametrises those subschemes of $X_d$
which are anticanonically embedded del Pezzo surfaces of degree $d$. This exists by a 
general result of Grothendieck \cite[Expos\'{e} 221, Thm.~3.1]{FGA}, and commutes with arbitrary
base extensions. Given a $\ZZ$-scheme $A$, we denote by $$\mathcal{H}_{d,A} = \mathcal{H}_d \times_{\Spec \ZZ} A$$
the Hilbert scheme of anticanonically embedded del Pezzo surfaces of degree $d$ over $A$.
We study the basic properties of these schemes in the following lemma.

\begin{lemma} \label{lem:connected}
	Let $1 \leq d \leq 9$ and let $k$ be a field. Then the scheme $\mathcal{H}_{d,k}$ is smooth. When $d \neq 8$, it is geometrically connected.
	When $d=8$, it consists of two connected components, each of which is geometrically connected.
\end{lemma}
\begin{proof}
	Both claims are clear when $d\leq 3$, as  $\mathcal{H}_{d,k}$
	is an open subset of some projective space (it parametrises
	the smooth members of a linear system of divisors).
	
	Next consider the case $d=4$. Here any such surface is a complete intersection
	of two quadrics, which is uniquely determined by the pencil of quadrics in $\PP^4$
	which contains it (namely it is the base-locus of such a pencil).
	Hence $\mathcal{H}_{4,k}$ may be identified with an open subset of the 
	Grassmanian of lines in the projective space $\PP(H^0(\PP^4,\mathcal{O}(2)))$, 
	hence is clearly smooth and geometrically connected.
	
	Assume now that $d\geq 5$. In order to prove that $\mathcal{H}_{d,k}$
	is smooth, by \cite[Expos\'{e} 221, Cor.~5.4]{FGA} it suffices to 
	show that $\smash{H^1(S, N_{S/\PP^d})=0}$ for any smooth del Pezzo surface
	$S \subset \PP^d$ of degree $d$, where $\smash{N_{S/\PP^d}}$ denotes the normal bundle
	of $S$ with respect to $\PP^d$. This vanishing is a standard cohomological calculation,
	and is shown, for example, in \cite[Lem.~5.7]{Cos06}.

	We now consider the connected components when $d \geq 5$.
	First suppose that $d \neq 8$. Here there is a unique
	del Pezzo surface of degree $d$ over $\bar{k}$ up to isomorphism.
	Thus any two del Pezzo surfaces of degree $d$
	in $\PP^d$ are projectively equivalent over $\bar{k}$, hence $\mathcal{H}_{d,k}$ is a 
	homogeneous space for $\PGL_{d+1}$, and the result follows.
	When $d=8$, there are exactly two $\smash{\bar{k}}$-isomorphism
	classes of del Pezzo surfaces of degree $8$, and each contains a surface
	defined over $k$. The result then follows from a similar argument to the above.
\end{proof}
We denote the two connected components of $\smash{\mathcal{H}_{8,k}}$ by 
$\smash{\mathcal{H}_{8,k}^1}$ and $\smash{\mathcal{H}_{8,k}^2}$. We choose these such that
$\smash{\mathcal{H}_{8,k}^1}$ parametrises twists of $\smash{\PP^1 \times \PP^1}$.

Next, recall (see \cite[\S 9]{Ser97a} or \cite[\S 3]{Ser08}) that given a variety $X$ over $k$,
a subset $\Omega \subset X(k)$ is said to be \emph{thin} if it is a finite union of 
subsets which are either contained in a proper closed subvariety of $X$, or 
in some $\pi(Y(k))$ where $\pi: Y \to X$ is a generically finite dominant morphism 
of degree exceeding $1$, with $Y$ irreducible. 
Given a disjoint union of varieties $X_1 \sqcup X_2$, we shall
say that a subset $\Omega \subset (X_1 \sqcup X_2)(k)$ is \emph{thin}, if $\Omega \cap X_i(k)$
is thin for each $i=1,2$.

\begin{theorem} \label{thm:distribution}
	Let $k$ be a number field and let $1 \leq d \leq 9$.
	\begin{enumerate}
		\item[(a)] There exists a thin subset $\Omega_d \subset \mathcal{H}_{d,k}(k)$ such that the Hasse principle
			for lines holds for those del Pezzo surfaces corresponding to the points
			of $\mathcal{H}_{d,k}(k) \setminus \Omega_d$.
		\item[(b)]The collection of del Pezzo surfaces
			of degree $d$ in $\mathcal{H}_{8,k}^1(k)$ which fail the Hasse principle
			for lines is Zariski dense in $\mathcal{H}_{8,k}^1$.
		\item[(c)] If $d=5,3,2$ or $1$, then the collection of del Pezzo surfaces
			of degree $d$ in $\mathcal{H}_{d,k}(k)$ which fail the Hasse principle
			for lines is Zariski dense inside $\mathcal{H}_{d,k}$. 
	\end{enumerate}
\end{theorem}

This result says that, whilst there are many counter-examples,
they are still in some sense quite rare.

\subsection{The universal family of lines}
Before we begin the proof of Theorem~\ref{thm:distribution},
we study the universal family of lines over the
Hilbert schemes $\mathcal{H}_d$.

The scheme $\mathcal{H}_d$ comes equipped with a universal
family of del Pezzo surfaces $\mathcal{S}_d \to \mathcal{H}_d$,
which naturally lives inside $\mathcal{H}_d \times X_d$.
By a standard argument (see \cite[\S 3.4]{Sch85}), there
exists a universal family of lines $\ell_d: \mathcal{L}_d \to \mathcal{H}_d$,
parametrising the lines in the family $\mathcal{S}_d \to \mathcal{H}_d$.
When $d=8$ and $k$ is a field, we have $\mathcal{H}_{8,k} = \mathcal{H}_{8,k}^1 \sqcup \mathcal{H}_{8,k}^2$
by Lemma \ref{lem:connected}.
In which case we denote by $\smash{\ell_{8,k}^1: \mathcal{L}_{8,k}^1 \to \mathcal{H}_{8,k}^1}$ and 
$\smash{\ell_{8,k}^2: \mathcal{L}_{8,k}^2 \to \mathcal{H}_{8,k}^2}$ the corresponding universal 
families of lines.

\begin{lemma} \label{lem:universal}
	Let $1 \leq d \leq 9$ and let $k$ be a field. Then the 
	map $$\ell_{d,k}:\mathcal{L}_{d,k} \to \mathcal{H}_{d,k},$$
	is smooth and projective. If $d\leq 7$, then it is finite \'{e}tale.
\end{lemma}
\begin{proof}
	By construction the morphism $\ell_{d,k}$ is projective.
	Since $\mathcal{H}_{d,k}$ is reduced by Lemma \ref{lem:connected},
	we may apply \cite[Thm.~III.9.9]{Har77} to each irreducible component 
	of $\mathcal{H}_{d,k}$ to deduce that $\ell_{d,k}$ is flat.
	Proposition \ref{prop:Hilbert} implies that $\ell_{d,k}$ is smooth
	and quasi-finite when $d \leq 7$.
	The result then follows since a quasi-finite smooth proper
	morphism is finite \'{e}tale (see also \cite[Prop.~3.6]{Sch85}).
\end{proof}

\begin{proposition} \label{prop:universal}
	Let $1 \leq d \leq 9$ and let $k$ be field, with
	$\characteristic k = 0$ or $\characteristic k \gg 1$. If $d\neq 7,8$, 
	then $\mathcal{L}_{d,k}$ is irreducible. When $d=8$, each 
	of the schemes $\mathcal{L}_{8,k}^1$ and $\mathcal{L}_{8,k}^2$ is irreducible.
\end{proposition}
\begin{proof}
	By a spreading out argument, it suffices to prove the result
	when $k$ has characteristic $0$.	
	For $d = 9,6$ or $5$, we may assume that $k$ is algebraically closed,
	and use a similar method to Lemma \ref{lem:connected}.
	Namely, not only is there a unique del Pezzo surface of degree $d$ over $k$, 
	but moreover the automorphism group of the surface acts transitively
	on the lines (for $d=9$ this is trivial, and for $d=6,5$
	this follows from \cite[Prop.~8.2.3]{Dol12}, \cite[Thm.~8.4.2]{Dol12} and \cite[Thm.~8.5.6]{Dol12}).
	As each automorphism of the surface is induced by an automorphism
	of the ambient projective space, it follows that $\mathcal{L}_{d,k}$ is a homogeneous
	space for $\PGL_{d+1}$, hence is irreducible. A similar argument applies to give the
	result when $d=8$.
	For $d=4$ or $3$, the result is a special case of 
	general results on Hilbert schemes of linear subspaces
	on complete intersections; see for instance \cite[p.~3]{DM98}.
	
	We now consider the case where $d=2$ or $1$. 
	Here, it suffices to prove the result when $k = \bar{\QQ}$.
	However to do this, we may reduce to the case where 
	$k$ is an arbitrary number field. As $\ell_{d,k}$ is proper,
	it suffices to show that there exists
	a single fibre which is irreducible, i.e.~we need a del Pezzo
	surface $S$ of degree $d$ over $k$ whose Hilbert scheme of lines $\mathcal{L}(S)$
	is irreducible. This problem has been considered by numerous authors,
	and the relevant results already exist in the literature.
	When $d = 1$, the required result over $\QQ$ is \cite[Thm.~1.3]{VAZ09},
	and moreover \cite[Rem.~1.4]{VAZ09} gives the result over any number
	field. When $d = 2$, the result over $\QQ$ is \cite{Ern94}.
	The method of \cite{VAZ09} is based on the approach used in \cite{Ern94}
	however, and a minor modification of the argument given at the end of 
	\cite[\S 6]{VAZ09}, which we do not reproduce here, yields the result over any number field.	
	Both of these results can also be deduced from work of Shioda \cite{Shi91}.
	This completes the proof.
\end{proof}

\begin{remark}
	Note that $\mathcal{L}_7$ is not irreducible. Indeed, every del Pezzo surface of degree $7$
	has a distinguished line, thus the morphism $\mathcal{L}_7 \to \mathcal{H}_7$ admits
	a section.
\end{remark}

\subsection{Proof of Theorem \ref{thm:distribution}(a)}
We now assume that $k$ is a number field. Here we begin with the
following lemma.

\begin{lemma} \label{lem:irreducible}
	Let $1 \leq d \leq 6$ or $d = 8,9$. There exists a thin
	subset $\Omega_d~\subset~\mathcal{H}_{d,k}(k)$ such that the fibre of $\ell_{d,k}$ over
	every point of $\mathcal{H}_{d,k}(k) \setminus \Omega_d$ is irreducible.
\end{lemma}
\begin{proof}
	When $d=9$, the result is trivial (one may take $\Omega_d = \emptyset$ by Proposition~\ref{prop:Hilbert}).
	When $d\leq 6$, this follows immediately from Lemma \ref{lem:universal},
	Proposition \ref{prop:universal} and Hilbert's irreducibility theorem 
	\cite[Prop.~3.3.1]{Ser08}. When $d=8$, we note that the corresponding
	result is trivial for $\ell_{8,k}^2$. As $\ell_{8,k}^1$ is proper by Lemma \ref{lem:universal},
	we may consider the Stein factorisation 
	$$
	\xymatrix{\mathcal{L}_{8,k}^1  \ar[dr] \ar[r]^{\ell_{8,k}^1}& \mathcal{H}_{8,k}^1 \\
				& \mathcal{P}_{8,k}^1. \ar[u] } 
	$$
	As $\mathcal{L}_{8,k}^1$ is irreducible by Proposition \ref{prop:universal},
	we see that the scheme $\mathcal{P}_{8,k}^1$ is also irreducible. Moreover,
	by Proposition \ref{prop:Hilbert}, we know that the map $\smash{\mathcal{P}_{8,k}^1 \to \mathcal{H}_{8,k}^1}$
	is finite \'{e}tale of degree $2$. Therefore the result follows
	on applying Hilbert's irreducibility theorem to the morphism 
	$\mathcal{P}_{8,k}^1 \to \mathcal{H}_{8,k}^1$.
\end{proof}

Let $S$ be a del Pezzo surface of degree $d$ over $k$ which fails the Hasse
principle for lines. By Theorem \ref{thm:main} we know that $d=8,5,3,2$ or $1$,
and moreover $S$ must be a twist of $\PP^1 \times \PP^1$ when $d=8$ (see \S \ref{sec:DP8}).
Lemma \ref{lem:HP} and Lemma \ref{lem:C_1xC_2} therefore imply
that $\mathcal{L}(S)$ is reducible. Hence, the result follows from Lemma \ref{lem:irreducible}. \qed

\subsection{Proof of Theorem \ref{thm:distribution}(b)}
As explained in \S \ref{sec:DP8}, there exists
a del Pezzo surface $S$ of degree $8$ which 
is a twist of $\PP^1 \times \PP^1$ and which is
a counter-example to the Hasse principle for lines.
Choose an anticanonical embedding 
$S \subset \PP^8$ and identify $S$ with an element of $\smash{\mathcal{H}_{8,k}^1(k)}$.
As $\smash{\mathcal{H}_{8,k}^1}$ is a homogeneous space for $\PGL_9$ (see the proof
of Lemma \ref{lem:connected}),
we see that the orbit of $S$ under the action of $\PGL_9(k)$ is Zariski
dense in $\smash{\mathcal{H}_{8,k}^1}$, and the result follows. \qed

\subsection{Proof of Theorem \ref{thm:distribution}(c)}
We begin the proof with a lemma which says that, under suitable conditions, one may ``approximate''
a collection of points in $\PP^2$ over a finite field by a closed point whose residue
field is a pre-specified number field.

\begin{lemma} \label{lem:approximation}
	Let $L/k$ be a finite field extension of degree $r$,
	with Galois closure $K/k$. Let $\mathfrak{P}$ be a prime ideal
	of $K$ which is completely split over $k$
	and let $P_1,\ldots,P_r \in \PP^2(\FF_\mathfrak{P})$.
	Then there exists a closed point $P$ of $\PP^2_k$ with residue field $L$ such that
	$$ P \equiv P_1 \sqcup \cdots \sqcup P_r \mod \mathfrak{P}.$$
	Moreover, let $P'$ be a closed point of $\PP^2$ of degree $r'$ in general position and assume
	that $r + r' \leq 8$. Then we may choose $P$ such that $P \sqcup P'$
	lies in general position.
\end{lemma}
Note that in the statement of the lemma, we are viewing $P$ as a Galois invariant collection of 
points of $\PP^2(K)$. In particular, the reduction of $P$
modulo the prime ideal $\mathfrak{P}$ is a well-defined collection of elements of
$\PP^2(\FF_\mathfrak{P})$.
\begin{proof}
	For $i=1,\ldots,r$, let $\sigma_i:L \to K$  denote the distinct embeddings of $L$
	into the field $K$.
	We are looking for an element $Q \in \PP^2(L)$ such that
	\begin{equation} \label{eqn:sigma_i}
		\sigma_i(Q) \equiv P_i \mod \mathfrak{P}, \qquad \mbox{for all }i \in \{1,\ldots,r\}.
	\end{equation}
	For each $i$, choose an element $g_i \in \Gal(K/k)$
	such that $g_i \circ \sigma_i = \sigma_1$. It follows that \eqref{eqn:sigma_i} is
	equivalent to
	$$\sigma_1(Q) \equiv g_i(P_i) \mod g_i(\mathfrak{P}), \qquad \mbox{for all }i \in \{1,\ldots,r\}.$$
	Next let $\mathfrak{p}_i$ denote the distinct prime ideals of $L$ which lie below the $g_i(\mathfrak{P})$
	and let $Q_i \in \PP^2(\FF_{\mathfrak{p}_i})$ lie below the $g_i(P_i)$.
	We have reduced to finding a $Q$ such that
	$$Q \equiv Q_i \mod \mathfrak{p}_i, \qquad \mbox{for all }i \in \{1,\ldots,r\}.$$
	However, the existence of such a $Q$ is now implied by weak approximation
	for $\PP^2$ (see for example \cite[Thm.~5.1.2]{Sko01}), which implies that the map
	$$\PP^2(L) \to \prod_{i=1}^r \PP^2(\FF_{\mathfrak{p}_i}),$$
	is surjective. Moreover, simple topological considerations show
	that we may choose $Q$ such that the general position criterion in the
	statement of the lemma is satisfied when $r + r' \leq 8$. This completes the proof.
\end{proof}

We now use this to obtain the following.

\begin{lemma} \label{lem:approximation_2}
	Let $d=5,3,2$ or $1$.
	Then there exist infinitely many prime ideals~$\mathfrak{p}$
	of $k$ with the following property:
	
	Let $S \subset X_d$ be a split anticanonically embedded del Pezzo surface of degree $d$
	over $\FF_\mathfrak{p}$. Then there exists a smooth proper surface
	$\mathcal{S} \subset X_d$, defined over the localisation $\OO_{k,\mathfrak{p}}$ 
	of $\OO_k$ at $\mathfrak{p}$, whose special fibre is equal
	to $S$, and whose generic fibre is an anticanonically embedded del Pezzo surface
	of degree $d$ which is a counter-example to the Hasse
	principle for lines.
\end{lemma}
\begin{proof}
	We first construct $\mathcal{S}$ as an abstract surface,
	then show how to embed it into $X_d$. We do this in detail only
	in the case where $d=3$; the other cases being handled
	in a similar manner. We work with the case of $D_5$ given in 
	\S \ref{sec:D_5}, using the construction of
	\S \ref{sec:constructions}.	
	
	Choose a quintic extension $K_2$ of $k$,
	whose Galois closure $K$ is a $D_5$-extension of $k$ which satisfies
	the conditions of Lemma \ref{lem:solvable}. Let $K_1$ denote
	the unique quadratic subfield of $K$ and let $\mathfrak{p}$
	be a prime ideal of $k$ which is completely split in $K$ and whose norm $N(\mathfrak{p})$ is sufficiently large.
	Finally, let $S \subset X_3$ be a split anticanonically embedded del Pezzo surface of degree $3$
	over $\FF_\mathfrak{p}$. Blow-up an $\FF_\mathfrak{p}$-point
	of $S$ which does not lie on a line, to obtain a del Pezzo surface
	$S'$ of degree $2$ (such a point exists as $N(\mathfrak{p})$ is sufficiently large).
	Denote by $L$ the exceptional curve of the blow-up.
	Choose pairwise skew lines $L_1,\ldots,L_7$ on $S'$ such that
	$L_1$ and $L_2$ each intersect $L$ with multiplicity $1$, but
	$L_3,\ldots,L_7$ do not meet $L$. Contracting $L_1,\ldots,L_7$, we obtain
	a morphism $S' \to \PP^2$ together with $\FF_\mathfrak{p}$-points $P_1,\ldots,P_7$ given
	by the images of $L_1,\ldots,L_7$.
	
	By Lemma \ref{lem:approximation}, we may
	choose closed points $P$ and $Q$ of $\PP^2$
	with residue fields $K_1$ and $K_2$, respectively, such that $P \sqcup Q$ lies
	in general position and such that $P$ and $Q$ reduce to
	$P_1 \sqcup P_2$ and $P_3 \sqcup \cdots  \sqcup P_7$ modulo $\mathfrak{p}$, respectively.
	Let $\mathcal{S}'$ denote the surface over $\OO_{k,\mathfrak{p}}$
	given by the blow-up of $\PP^2$ in $P$ and $Q$.
	Contracting the line of $\mathcal{S}'$ which reduces to $L$ modulo $\mathfrak{p}$,
	we obtain a surface $\mathcal{S}$ whose special fibre is isomorphic
	to $S$. As explained in \S \ref{sec:D_5}, the generic fibre of $\mathcal{S}$ is a counter-example
	to the Hasse principle for lines, as required.
	This works for any prime ideal $\mathfrak{p}$ which is completely split
	in $K$ and of sufficiently large norm, in particular, there are infinitely 
	many such prime ideals by Chebotarev's density theorem.
	
	When $d=5,2$ or $1$, the results obtained in \S \ref{sec:DP5}, \S \ref{sec:DP2}
	and \S \ref{sec:DP1} imply that we may obtain counter-examples
	to the Hasse principle for lines by blowing-up certain collections
	of closed points in $\PP^2$ with special residue fields.
	Applying a similar  method to the above, we obtain the required surface
	$\mathcal{S}$ in these cases (in fact here it is even easier,
	since no contraction is required).
	
	We now show that the embedding of $S$ can be ``lifted'' to an embedding of $\mathcal{S}$.
	To prove this, let $\pi:\mathcal{S} \to \Spec R$ be as constructed above, where $R=\OO_{k,\mathfrak{p}}$.
	As $R$ is affine, there is a natural isomorphism
	$$H^0(R,\pi_* \omega^{-n}_{\mathcal{S}/R}) \cong H^0(\mathcal{S}, \omega^{-n}_{\mathcal{S}/R}),$$
	of $R$-modules, for each $n \geq 0$.
	Hence by Grauert's theorem \cite[Cor.~III.12.9]{Har77}, the 
	reduction modulo $\mathfrak{p}$ map
	$$ H^0(\mathcal{S}, \omega^{-n}_{\mathcal{S}/R}) \to H^0(S, \omega^{-n}_{S/\FF_{\mathfrak{p}}})$$
	is surjective. As the anticanonical embedding of $S$
	is determined by a collection of global sections of $H^0(S, \omega^{-n}_{S/\FF_{\mathfrak{p}}})$
	for suitable $n \geq 0$,
	it follows that any anticanonical embedding $S \subset X_d$ 
	lifts to an anticanonical embedding $\mathcal{S} \subset X_d$, as required.
\end{proof}

Let $\mathfrak{p}$ be a prime ideal of $k$. Denote by $C(\OO_{k,\mathfrak{p}})$ 
the set of elements of $\mathcal{H}_{d}(\OO_{k,\mathfrak{p}})$
whose generic fibres are counter-examples
to the Hasse principle for lines. Let $C(k) \subset \mathcal{H}_{d}(k)$
and $C(\FF_\mathfrak{p}) \subset \mathcal{H}_{d}(\FF_\mathfrak{p})$
denote the generic fibres and special fibres of these
elements, respectively. Assume that $C(k)$ is not Zariski dense in $\mathcal{H}_{d,k}(k)$.
Applying the Lang-Weil estimates \cite{LW54}, we see that 
$$\# C(\FF_\mathfrak{p}) \ll N(\mathfrak{p})^{\delta(d)-1}, \quad \mbox{ as } N(\mathfrak{p}) \to \infty,$$
where
$$\delta(d) = \dim \mathcal{H}_{d,k}.$$
However by Lemma \ref{lem:approximation_2}, we know that there are infinitely many 
prime ideals $\mathfrak{p}$ of $k$ such that
$$\mathcal{H}_d(\FF_\mathfrak{p})_{\mathrm{split}} \subset C(\FF_\mathfrak{p}),$$
where $\mathcal{H}_d(\FF_\mathfrak{p})_\mathrm{split}$ denotes the collection of elements of 
$\mathcal{H}_d(\FF_\mathfrak{p})$ which correspond to split surfaces. Therefore to complete the proof
of Theorem \ref{thm:distribution}(c), it suffices to show the following.

\begin{lemma}
	Let $1 \leq d \leq 7$ and let $q$ be a power of a sufficiently large prime.
	Then
	$$\# \mathcal{H}_d(\FF_q)_{\mathrm{split}} \gg q^{\delta(d)}, \quad \mbox{ as } q \to \infty.$$
\end{lemma}
\begin{proof}
	By the definition of $\mathcal{L}_d$, we have
	$$\mathcal{H}_d(\FF_q)_{\mathrm{split}} = \ell_d( \mathcal{L}_d(\FF_q)).$$
	As $\ell_{d,\FF_q}$ is finite \'{e}tale by Lemma \ref{lem:universal}, it therefore
	suffices to show that 
	$$\# \mathcal{L}_d(\FF_q) \gg q^{\delta(d)}, \mbox{ as } q \to \infty.$$
	However again since $\ell_{d,\FF_q}$ is finite \'{e}tale, we have
	$\dim \mathcal{L}_{d,\FF_q} = \delta(d)$ for $q \gg 1$.
	As $\mathcal{L}_{d,\FF_q}$ is geometrically integral for $q$
	sufficiently large by Proposition \ref{prop:universal},
	the result follows from applying the Lang-Weil estimates \cite{LW54}
	to $\mathcal{L}_{d,\FF_q}$.	
\end{proof}
This completes the proof of Theorem \ref{thm:distribution}(c), hence the proof of Theorem \ref{thm:distribution}. \qed

\subsection{Proof of Theorem \ref{thm:thin_cubic}}
Here $\mathcal{H}_3 \subset \PP^{19}$ is the open subset given by the complement
of the discriminant locus. Theorem \ref{thm:thin_cubic}(a) follows from
Theorem \ref{thm:distribution}(c).
Next by Theorem \ref{thm:distribution}(a), we know that 
the counter-examples to the Hasse principle for lines lie on a thin subset.
Theorem \ref{thm:thin_cubic}(b) therefore follows from the fact that $0\%$ of rational points
of bounded height lie inside any given thin subset of projective space
(see \cite[Thm.~13.3]{Ser97a}).
The result is proved. \qed

\section{Intersections of two quadrics} \label{sec:conceptual}
The aim of this section is to prove Theorem \ref{thm:two_quadrics}.
We begin with some results on the geometry of intersections of two
quadrics and pencils of quadrics. For background see \cite{Rei72} 
or \cite[\S 3.3]{Wit07}.

\subsection{Geometry of intersections of two quadrics} \label{sec:geometry}
Let $n \geq 0$ be an integer and let $k$ be a perfect field
of characteristic not equal to $2$. Let $X$ be a smooth complete intersection of two quadrics
over $k$ in $\PP^{2n+2}$. Denote by $\mathcal{L}(X)$ the Hilbert scheme
of $n$-planes inside $X$, which is a finite scheme of degree $2^{2n+2}$. Note
that when $X$ is a del Pezzo surface, this agrees with the definition 
given in \S \ref{sec:lines}. For these facts, and more, see \cite{Rei72}.

The aim of this section is to construct the action of a certain
group scheme  $\Res_{K/k}(\mu_2)/ \mu_2$ on $X$, where $\Res_{K/k}$ denotes the Weil restriction,
whose induced action on $\mathcal{L}(X)$ is simply transitive
(i.e.~such that $\mathcal{L}(X)$ is a torsor under $\Res_{K/k}(\mu_2)/ \mu_2$).
We shall do this by using pencils of quadrics
and descent. The constructions we shall use are generalisations
of some of the constructions used in \cite{Sko10}, where the author only considers the case
of del Pezzo surfaces of degree~$4$. Suppose that $X$ is defined by the equations
$$Q_1(x) = Q_2(x) = 0 \quad \subset \PP^{2n+2}.$$
Consider the pencil of quadrics containing $X$,
given as the subvariety
$$\lambda Q_1(x) + \mu Q_2(x) = 0 \quad \subset \PP^1 \times \PP^{2n+2}.$$
To this we associate the discriminant
$$\det(\lambda Q_1 + \mu Q_2) =0 \quad \subset \PP^1,$$
which parametrises the singular quadrics in the pencil. Note that here
we abuse notation, and identify each $Q_i$ with the corresponding symmetric matrix.
As $X$ is smooth, the discriminant is a finite reduced closed subscheme of $\PP^1$
of degree $2n+3$ \cite[Prop.~3.26]{Wit07}. Let $K$ denote the associated
finite \'{e}tale $k$-algebra.

In order to construct the associated group scheme, we first work over $\bar k$.
Each singular quadric in the pencil over $\bar k$ is a cone over a smooth quadric
of dimension~$2n$, and moreover the singular point of the quadric
does not lie on $X_{\bar k}$ \cite[Lem.~3.7]{Wit07}.
Projecting away from the singular point,
one realises $X_{\bar k}$ as a double cover of the associated smooth quadric. The deck
transformation of this cover is then an involution defined on $X_{\bar k}$.
Considering the group generated by these involutions, we obtain an action of the group $(\ZZ/2\ZZ)^{2n+3}$
on $X_{\bar k}$. All this data is Galois invariant, hence descends to $k$. 
We therefore obtain the action of a group scheme on $X$ over $k$, which
is easily seen to be isomorphic to 
$\Res_{K/k}(\mu_2)$. This does not act faithfully on $X$,
however the quotient $\Res_{K/k}(\mu_2)/\mu_2$ does act faithfully.

\begin{proposition} \label{prop:torsor}
	Let $X$ be a smooth $2n$-dimensional complete intersection of two quadrics over $k$
	and let $\Res_{K/k}(\mu_2)/ \mu_2$ be the associated group scheme, as constructed above.
	Then $\mathcal{L}(X)$ is a torsor under $\Res_{K/k}(\mu_2)/ \mu_2$.
\end{proposition}
\begin{proof}
	To prove the result we may assume that $k$ is algebraically closed.
	In which case, this is \cite[Thm.~3.8]{Rei72}.
\end{proof}

\subsection{Proof of Theorem \ref{thm:two_quadrics}}
\subsubsection{Tate-Shafarevich groups}
We begin the proof by recalling some facts about Tate-Shafarevich groups
of commutative group schemes. Let $k$ be a number field and $S$ be a finite set of places
of $k$.
For a commutative group scheme $G$ over $k$, denote by
$$\Sha_S(k,G) = \ker\left(H^1(k,G) \to \prod_{\mathclap{\substack{v \in \Val(k) \\ v \not \in S}}} H^1(k_v, G)\right),$$
the \emph{$S$-Tate-Shafarevich group} of $G$. Here $H^1(k,G)$ classifies
torsors of $G$ up to isomorphism, with a torsor having the trivial
class if and only if it has a rational point (see for example \cite[\S 2.1]{Sko01}). In particular, the non-trivial
elements of $\Sha_S(k,G)$ classify those torsors of $G$ which are locally
soluble outside of $S$. When $S=\emptyset$, such torsors are exactly those which 
fail the Hasse principle.

\subsubsection{Proof of Theorem \ref{thm:two_quadrics}}
Given Proposition \ref{prop:torsor}, in order to prove Theorem~\ref{thm:two_quadrics}
it suffices to show the following.

\begin{proposition} \label{prop:Sha}
	Let $k$ be a number field and let $S$ be a finite set of places of~$k$.
	Let $K$ be an \'{e}tale $k$-algebra of odd degree.
	Then
	$$\Sha_S(k,\Res_{K/k}(\mu_2) / \mu_2) = 0.$$
\end{proposition}
\begin{proof}
	First note that since $K/k$ has odd degree, the composed
	morphism $$\Res_{K/k}^1 (\mu_2) \to \Res_{K/k}(\mu_2) \to \Res_{K/k}(\mu_2)/ \mu_2,$$
	is an isomorphism, where $\Res_{K/k}^1(\mu_2)$ denotes the norm $1$ subgroup scheme of $\Res_{K/k}(\mu_2)$.
	From which it follows that the sequence
	$$
		1 \to \mu_2 \to \Res_{K/k}(\mu_2) \to \Res_{K/k}(\mu_2)/\mu_2 \to 1,
	$$
	is a split short exact sequence of group schemes. Hence we obtain an embedding
	$$\Sha_S(k,\Res_{K/k}(\mu_2) / \mu_2) \hookrightarrow \Sha_S(k,\Res_{K/k}(\mu_2)).$$
	Write $K=\prod_{i=1}^r K_i$, where the $K_i$ are field extensions of $k$
	and let $S_i$ denote the set of places of $K_i$ which lie above the places of $S$.
	Then by Shapiro's lemma \cite[Prop.~2.5.10]{Ser97b} it suffices to show that 
	$\Sha_{S_i}(K_i,\mu_2)=0$ for each $i \in \{1,\ldots,r\}$. However $H^1(K_i,\mu_2)$
	classifies quadratic extensions of $K_i$, hence this vanishing follows from the
	fact that any quadratic extension is non-split at infinitely many places.
	This proves the result.
\end{proof}
This completes the proof of Theorem \ref{thm:two_quadrics}. \qed


\begin{thebibliography}{xx}
	\bibitem{BC14} {B.~Banwait and J.~Cremona}, 
	{Tetrahedral elliptic curves and the local-to-global principle for isogenies}.
	{\em Algebra and Number Theory} {\bf 8} (2014), no.~5, 1201--1229.

	\bibitem{BB96} {D.~Berend and Y.~Bilu}, {Polynomials with roots modulo every integer}.
	{\em Proc. AMS. } {\bf 124} (1996), no.~6, 1663--1671.
	
	\bibitem{Bir57} {B.~J.~Birch}, {Homogeneous forms of odd degree in a large number of variables}.
	{\em Mathematika} {\bf 4} (1957), no. 2, 102--105.
	
	\bibitem{BCP97} {W. Bosma, J. Cannon and C. Playoust}, 
	{The Magma algebra system. I. The user language.}
	{\em J. Symbolic Comput.}, {\bf 24} (1997), no.~3-4, 235--265.
	
	\bibitem{Bou81}
	{N.~Bourbaki}, \textit{\'{E}l\'{e}ments de math\'{e}matique. Alg\`{e}bre}.
	Chapitres 4 \`{a} 7. Lecture Notes in Mathematics, 864. Masson, Paris, 1981. 
	
	\bibitem{BK14} {J.~Bourgain and A.~Kontorovich}, 
	{On the local-global conjecture for integral Apollonian gaskets}.
	{\em Invent.~Math.} {\bf 196} (2014), no.~3, 589--650.
	
	\bibitem{Bra14} {J.~Brandes}, {Forms representing forms and linear spaces on hypersurfaces}.
	{\em Proc. Lond. Math. Soc.} {\bf 108} (2014), no.~4, 809--835.
	
	\bibitem{BMS14} {T.~D.~Browning, L.~Matthiesen and A.~Skorobogatov}, 
	{Rational points on pencils of conics and quadrics with many degenerate fibres}.
	{\em Annals of Math.} {\bf 180} (2014), no.~1, 381--402.
	
	\bibitem{CTSSD87} {J.-L.~Colliot-Th\'{e}l\`{e}ne, J.-J.~Sansuc and P.~Swinnerton-Dyer}, 
	{Intersections of two quadrics and Ch\^{a}telet surfaces. II.}
	{\em J. Reine Angew. Math.} {\bf374} (1987), 72--168.
	
	\bibitem{CM96} {D.~Coray and C.~Manoil},
	{On large Picard groups and the Hasse principle for curves and K3 surface}.
	{\em Acta Arith}. {\bf76} (1996), no.~2, 165--189.
	
	\bibitem{Cos06} {I.~Coskun},
	{The enumerative geometry of Del Pezzo surfaces via degenerations}.
	{\em Amer. J. Math.} {\bf128} (2006), no.~3, 751--786.
	
	\bibitem{Die10} {R.~Dietmann}, {Linear spaces on rational hypersurfaces of odd degree}.
	{\em Bull. Lond. Math. Soc.} {\bf 42} (2010), no.~5, 891--895.
	
	\bibitem{DW03} {R.~Dietmann and T.~Wooley}, {Pairs of cubic forms in many variables}.
	{\em Acta Arith.} {\bf 110} (2003), no.~2, 125--140.
		
	\bibitem{Dol12} {I.~Dolgachev}, \textit{Classical algebraic geometry. A modern view}.
	Cambridge University Press, Cambridge, 2012.
	
	\bibitem{DM98} {O. Debarre and L. Manivel}, 
	{Sur la vari\'{e}t\'{e} des espaces lin\'{e}aires contenus dans une intersection compl\`{e}te}.
	{\em Math. Ann.} {\bf312} (1998), no.~3, 549--574.
	
	\bibitem{EJ10} {A.-S.~Elsenhans and J.~Jahnel}, 
	{Cubic surfaces with a Galois invariant double-six}.
	{\em Central European Journal of Mathematics} {\bf8} (2010), no.~4, 646--661.
	
	
	\bibitem{Ern94} {R.~Ern\'{e}}, 
	{Construction of a del Pezzo surface with maximal Galois action on its Picard group}.
	{\em J. Pure Appl. Algebra} {\bf 97} (1994), no.~1, 15--27.
		
	\bibitem{FGA} {A.~Grothendieck}, 
	{\em Fondements de la g\'{e}om\'{e}trie alg\'{e}brique}.
	[Extraits du S\'{e}minaire Bourbaki, 1957--1962]. Secr\'{e}tariat math\'{e}matique,
	Paris 1962.
	
	\bibitem{Gei69} {C.~F.~Geiser}, 
	{Ueber die Doppeltangenten einer ebenen Curve vierten Grades}.
	{\em Math. Ann.} {\bf 1} (1869), no.~1, 129--138.
	
	\bibitem{Har77} {R.~Hartshorne}, \textit{Algebraic Geometry}.
	Springer-Verlag, 1977.	
	
	\bibitem{Isk70} {V.~A.~Iskovskih}, 
	{Rational surfaces with a sheaf of rational curves and with a positive square of canonical class}.
	{\em Mat. Sb. (N.S.)} {\bf 83}(125) (1970), 90--119.
	
	\bibitem{Isk79} {V.~A.~Iskovskih}, 
	{Minimal models of rational surfaces over arbitrary fields}.
	{\em Izv. Akad. Nauk SSSR Ser. Mat.} {\bf 43} (1979), no.~1, 19--43, 237.
	
	\bibitem{Ked12} {K.~S.~Kedlaya}, 
	{A construction of polynomials with squarefree discriminants}.
	{\em Proc. AMS.} {\bf 140} (2012), no.~12, 3025--3033.
	
	\bibitem{KM01} {J. Kl\"{u}ners, G. Malle}, 
	{A database for field extensions of the rationals}.
	{\em LMS. J. Comput. Math.} {\bf 4} (2001), 182--196.
	
	\bibitem{Kol96} {J.~Koll\'{a}r}, 
	{\em Rational curves on algebraic varieties}.
	A Series of Modern Surveys in Mathematics 
	 32. Springer-Verlag, Berlin, 1996.
	
	\bibitem{Kol97} {J.~Koll\'{a}r}, 
	{Polynomials with integral coefficients, equivalent to a given polynomial}.
	{ \em Electron. Res. Announc. AMS.} {\bf 3} (1997), 17--27.
	
	\bibitem{Lan02} {S.~Lang}, {\em Algebra}. 
	Revised third edition. Graduate Texts in Mathematics, {\bf 211}.
	Springer-Verlag, New York, 2002.
	
	\bibitem{LW54} {S.~Lang and A.~Weil},
	{Number of points of varieties in finite fields}.
	{\em Amer. J. Math.} {\bf76}, no.~4, (1954), 819--827
	
	\bibitem{Lee84} {D.~B.~Leep}, 
	{Systems of quadratic forms.}
	{\em J. Reine Angew. Math.} {\bf350} (1987), 109--116.
	
	\bibitem{Liu02} {Q.~Liu}, {\em Algebraic geometry and arithmetic curves}. 
	Oxford Graduate Texts in Mathematics, {\bf6}.
	Oxford Science Publications. Oxford University Press, Oxford, 2002.
	
	\bibitem{Man86} {Y.~I.~Manin}, {\em Cubic Forms: Algebra, geometry, arithmetic}. North-Holland
	Mathematical Library {\bf4}, North-Holland Publishing Co., 2nd ed. 1986.
	
	\bibitem{Mar97} {G.~Martin}, 
	{Solubility of systems of quadratic forms.}
	{\em Bull. Lond. Math. Soc.} {\bf29} (1997), no.~4, 385--388.
	
	\bibitem{Nar04} {W.~Narkiewicz}, {\em Elementary and analytic theory of algebraic numbers}. 
	Third edition. Springer Monographs in Mathematics. Springer-Verlag, Berlin, 2004.
		
	\bibitem{Neu99} {J.~Neukirch}, {\em Algebraic number theory}. Springer-Verlag, Berlin, 1999.
	
	\bibitem{O'Me73} {O.~T.~O'Meara}, 
	{\em Introduction to Quadratic Forms}. 
	Corr. 3rd printing.
	Springer-Verlag, Berlin, 1973.	
	
	\bibitem{Rei72} {M.~Reid},
	{\em The complete intersection of two or more quadrics}. Ph.D. thesis.
	
	\bibitem{Ser80} {J.-P.~Serre}, 
	{\em Local Fields}. Springer-Verlag, Berlin, 1980.	
	
	\bibitem{Ser97a}
	{J.-P.~Serre}, {\em Lectures on the Mordell--Weil theorem}.
	3rd ed., Aspects of Mathematics, Friedr.\ Vieweg \& Sohn, Braunschweig, 1997.
	
	\bibitem{Ser97b} {J.-P.~Serre}, {\em Galois cohomology},
	corrected ed. Springer, 1997.
	
	\bibitem{Ser03} {J.-P. Serre}, {On a theorem of Jordan}.
	{\em Bull. AMS.} {\bf 40} (2003), no.~4, 429 --440.
	
	\bibitem{Ser08} {J.-P.~Serre}, {\em Topics in Galois theory}.
	Second edition.
	Research Notes in Mathematics~1. A K Peters, Ltd., Wellesley, MA, 2008. 
	
	\bibitem{Sch85} {A.-J.~Scholl}, {A finiteness theorem for del Pezzo surfaces 
	over algebraic number fields}.
	{\em J. Lond. Math. Soc.} {\bf 32} (1985), no.~1, 31--40. 
		
	\bibitem{Shi91}
	{T.~Shioda}, {Theory of Mordell-Weil lattices}.
	{\em Proceedings of the International Congress of Mathematicians}, 
	Vol. I, II (Kyoto, 1990), 473--489, Math. Soc. Japan, Tokyo, 1991.
	
	\bibitem{Sko01} {A.~Skorobogatov}, {\em Torsors and rational points}.
	Cambridge University press, 2001.
	
	\bibitem{Sko10}
	{A.~Skorobogatov}, {Del Pezzo surfaces of degree $4$ and their relation to Kummer surfaces}.
	{\em Enseign. Math.}, {\bf 56} (2010), no.~1--2, 73--85.
	
	\bibitem{Son08} {J.~Sonn}, {Polynomials with roots in ${\QQ}_p$ for all $p$}.
	{\em Proc. AMS.} {\bf 136} (2008), no.~6, 1955--1960.
	
	\bibitem{Son09} {J.~Sonn}, {Two remarks on the inverse Galois problem 
	for intersective polynomials}.
	{\em J. Th. Nombres de Bordeaux} {\bf 21} (2009), no.~2, 437--439.
	
	\bibitem{Sut12} {A.~Sutherland}, 
	{A local-global principle for rational isogenies of prime degree}.
	{\em J. Th. Nombres de Bordeaux} {\bf 24} (2012), no.~2, 475--485.
		
	\bibitem{Sto07} {M.~Stoll}, {Finite descent obstructions and rational points on curves}.
	{\em Algebra and Number Theory} {\bf 1} (2007), no.~4, 349--391.
	
	\bibitem{VAZ09} {A.~V\'{a}rilly-Alvarado and D.~Zywina}, 
	{Arithmetic $E_8$ lattices with maximal Galois action}.
	{\em LMS J. Comput. Math.} 
	{\bf 12} (2009), 144 --165.
		
	\bibitem{Wit07} {O.~Wittenberg}, 
	{\em Intersections de deux quadriques et pinceaux de courbes de genre 1}.
	Lecture Notes in Mathematics, 1901. Springer, Berlin, 2007.
\end{thebibliography}
\end{document}